\documentclass[12pt,twoside,a4paper]{article}

\usepackage{amsmath,mathrsfs,euscript,amssymb,amsthm,a4,times}
\usepackage{color}
\usepackage{verbatim}

\newtheorem{theorem}{Theorem}
\numberwithin{theorem}{section}
\newtheorem{lemma}[theorem]{Lemma}
\newtheorem{corollary}[theorem]{Corollary}
\newtheorem{proposition}[theorem]{Proposition}

\newtheorem{definition}[theorem]{Definition}

\newtheorem{remark}[theorem]{Remark}

\DeclareMathOperator{\diam}{diam}

\DeclareMathOperator{\spn}{span}
\DeclareMathOperator{\epi}{epi}
\DeclareMathOperator{\dom}{dom}

\newcommand{\abs}[1]{|#1|}

\newcommand{\dprb}[1]{\big\langle #1 \big\rangle}

\newcommand{\ddpr}[1]{\langle\!\langle #1 \rangle\!\rangle}

\newcommand{\cl}[1]{\overline{#1}}
\newcommand{\di}{\mathrm{d}}
\newcommand{\dd}{\, \mathrm{d}}

\newcommand{\N}{\mathbb{N}}
\newcommand{\R}{\mathbb{R}}

\newcommand{\B}{\mathbb{B}}
\newcommand{\Hi}{\mathbb{H}}
\newcommand{\SH}{\mathbb{S}_{\mathbb{H}}}
\newcommand{\ONE}{\mathbf{1}}

\newcommand{\tows}{\overset{*}\rightharpoondown}

\newcommand{\toY}{\overset{\Yb}{\to}}

\newcommand{\LL}{\mathrm{L}}
\newcommand{\WW}{\mathrm{W}}

\newcommand{\CC}{\mathrm{C}}
\newcommand{\BV}{\mathrm{BV}}

\newcommand{\Yb}{\mathbf{Y}}

\newcommand{\Eb}{\mathbf{E}}

\newcommand{\Qb}{\mathbf{Q}}
\newcommand{\SQb}{\mathbf{SQ}}
\newcommand{\eps}{\varepsilon}

\newcommand{\term}[1]{\textbf{#1}}

\newcommand{\Lcal}{\mathcal{L}}

\newcommand{\Lscr}{\mathscr{L}}
\newcommand{\Mscr}{\mathscr{M}}

\newcommand{\real}{\mathbb{R}}
\newcommand{\Rn}{\mathbb{R}^n}
\newcommand{\RN}{\mathbb{R}^N}
\newcommand{\arn}{\real^n}
\newcommand{\card}{\#}
\renewcommand{\l}{\lambda}
\newcommand{\F}{\mathcal{F}}
\newcommand{\D}{\mathcal{D}}

\newcommand{\lip}{\operatorname{lip}}

\catcode`@=11
\newbox\tr@tto
\setbox\tr@tto=\hbox{{\count0=0\dimen0=-,9pt\dimen1=1,1pt\loop\ifnum\count0<11 \advance \count0 by1 \vrule width.51pt height\dimen1
     depth\dimen0\kern-0.17pt\advance\dimen0 by-0.05pt\advance\dimen1
     by0.1pt\repeat \loop\ifnum\count0<21\advance \count0 by1 \vrule
     width.6pt height\dimen1 depth\dimen0\kern-0.2pt \advance\dimen0
     by-0.1pt\advance\dimen1 by 0.05pt\repeat}}
\def\medint{\displaystyle\copy\tr@tto\kern-10.4pt\int}
\catcode`@=12

\numberwithin{equation}{section}

\newcommand{\mA}{\mathcal{A}}
\newcommand{\mV}{\mathcal{V}}
\newcommand{\mC}{\mathcal{C}}
\newcommand{\mD}{\mathcal{D}}
\newcommand{\mS}{\mathcal{S}}
\newcommand{\q}{\mathcal{Q}}

\title{ON RANK ONE CONVEX FUNCTIONS THAT ARE HOMOGENEOUS OF DEGREE ONE\footnote{Version: April 2015}}
\author{Bernd Kirchheim \, and \, Jan Kristensen}

\date{}

\begin{document}

\maketitle

\begin{abstract}
We show that positively $1$--homogeneous rank one convex functions are convex at $0$ and at matrices of
rank one. The result is a special case of an abstract convexity result that we establish for positively $1$--homogeneous
directionally convex functions defined on an open convex cone in a finite dimensional vector space.
From these results we derive a number of consequences including various generalizations of the Ornstein $\LL^1$ non 
inequalities. Most of the results were announced in ({\em C.~R.~Acad.~Sci.~Paris, Ser.~I 349 (2011), 407--409}).
\end{abstract}


\bigskip

\bigskip
\section{Introduction  and Statement of Results }  
\label{intro}

It is often possible to reformulate questions about sharp integral estimates for the derivatives of mappings 
as questions about certain semiconvexity properties of associated integrands (we refer the reader to
\cite{Dac89} for a survey of the relevant convexity notions and their roles in the calculus of
variations). Particularly fascinating examples of the utility of this viewpoint are presented in \cite{I02}, where 
the fact that rank one convexity is a necessary condition for quasiconvexity that is possible to check in
concrete cases, leads to a long list of tempting conjectures, all of which -- if proven -- would have significant impact
on the foundations of Geometric Function Theory in higher dimensions. The obstacle to success
is that  rank one convexity in general does not imply quasiconvexity. This negative result,
known as Morrey's conjecture \cite{Mo52}, was established in \cite{Sv92}.
It does, however, not exclude the possibility that some of these semiconvexity notions agree within 
more restricted classes of integrands having natural homogeneity properties. A very interesting case 
being the positively $1$--homogeneous integrands. Their semiconvexity properties correspond to 
$\LL^1$--estimates, and hence are difficult to establish using interpolation or other harmonic analysis
tools. In this paper we investigate the convexity properties of such integrands. In particular it is
shown in Corollary \ref{homorc1} that a positively $1$--homogeneous and rank one convex integrand must 
be convex at $0$ and at all matrices of rank one. While this class of integrands has been 
investigated several times before, see in particular \cite{Sv91,Mu92,DH96,DM07}, the surprising 
automatically improved convexity at all matrices of rank at most one remained unnoticed.  
We remark that \cite{Sv91} (and also \cite{Mu92} and \cite{DH96}) yield examples of $1$--homogeneous 
rank one convex functions that are {\em not convex} at some matrices of rank two so that these results 
are indeed sharp. The key convexity result is best stated in abstract terms, and we take a moment to introduce 
the requisite terminology (see also Section 2 for notation and terminology). 
Let $\mV$ be a finite dimensional real vector space, $\mC$ an open convex cone in $\mV$ and $\D$ a 
cone of directions in $\mV$. More precisely, $\mV$ is considered with some norm $\| \cdot \|$ and $\mC$ is
an open subset of $\mV$ with the property that $sx+ty \in \mC$ when $s$, $t>0$ and $x$, $y \in \mC$. The set
$\mD$ is also a subset of $\mV$, it gives the directions in which we have convexity, and we assume that 
$tx \in \D$ for all $x \in \mD$, $t\in \R$, and that it contains a basis for $\mV$. A real--valued function 
$f\colon \mC \to \R$ is $\mD$--convex provided its restrictions to line segments in $\mC$ in directions of $\mD$ are convex. 
The function $f$ is positively $1$--homogeneous provided $f(tx)=tf(x)$ for all $t>0$ and all $x \in \mC$. 
Finally we say that $f$ has linear growth at infinity if we can find a constant $c>0$ such that 
$\abs{f(x)} \leq c(\| x \| + 1)$ holds for all $x \in \mC$.




\begin{theorem}\label{convex}
Let $\mC$ be an open convex cone in a normed finite dimensional real vector space $\mV$, and $\mD$ a cone of
directions in $\mV$ such that $\mD$ spans $\mV$.

If $f \colon \mC \to \R$ is $\mD$--convex and positively $1$--homogeneous, then $f$ is convex at each point 
of $\mC \cap \mD$. More precisely, and in view of the homogeneity, for each $x_0 \in \mC \cap \mD$ there exists a linear 
function $\ell \colon \mV \to \R$ satisfying $\ell(x_{0})=f(x_{0})$ and $f \geq \ell$ on $\mC$.
\end{theorem}
We state separately the special case corresponding to rank one convexity for functions defined on the space of
real $N$ by $n$ matrices:

\begin{corollary}\label{homorc1}
A rank--one convex and positively $1$--homogeneous function $f\colon \R^{N \times n} \to \R$ is convex at each point of the
rank one cone $\{ x \in \R^{N \times n} : \, \mathrm{rank }\ x \leq 1 \}$. 
\end{corollary}

A remarkable result of Ornstein \cite{Or62} states that given a set
of linearly independent linear homogeneous constant--coefficient 
differential operators in $n$ variables of order $k$, say 
$Q_0$, $Q_{1}, \, \dots \, , \, Q_{m}$, and
any number $K>0$, there is a $\CC^{\infty}$ smooth function $\phi$ vanishing
outside the unit cube, $\phi \in \CC^{\infty}_{c}( (0,1)^{n})$, such that $\int \! |Q_{0}\phi | > K$ while 
$\int \! |Q_{j}\phi | < 1$ for all $1 \leq j \leq m$.

This result convincingly manifests the fact that estimates for differential
operators, usually based on Fourier multipliers and Calderon--Zygmund operators,
can be obtained for all $\LL^p$ with $p\in (1,\infty)$ by interpolation and 
(more directly) even for the weak-$\LL^1$ spaces but fail to extend to the 
limit case $p=1$. Ornstein used his result to answer a question by
L.~Schwarz by constructing a distribution in the plane that was not a 
measure but whose first order partial derivatives were distributions of first order. 
He then gave a very technical and rather concise proof of his statement for general 
dimension $n$ and degree $k$. 

Theorem \ref{convex}, and Corollary \ref{homorc1}, yield when combined with arguments
from the calculus of variations, various generalizations of Ornstein's 
$\LL^1$--non--inequality. In particular the approach allows also a streamlined and 
very elementary proof of an extension of Ornstein's result to the wider context of $x$--dependent 
and vector--valued operators:

\begin{theorem}\label{orne}
Let $\mathbb{V}$, $\mathbb{W}_1$, $\mathbb{W}_2$ be finite dimensional inner product spaces, and consider two $k$--th order
linear partial differential operators with locally integrable coefficients 
$$
a^{i}_{\alpha} \in \LL^{1}_{\rm loc}(\Rn , \Lscr (\mathbb{V},\mathbb{W}_{i})) \quad \quad (i=1,2)
$$ 
defined by
$$
A_{i}(x,D)\varphi := \sum_{\abs{\alpha}=k} a^{i}_{\alpha}(x)\partial^{\alpha}\varphi \quad \quad (i=1,2)
$$
for smooth and compactly supported test maps $\varphi \in \CC^{\infty}_{c}(\Rn , \mathbb{V})$.

Then there exists a constant $c$ such that
$$
\| A_{2}(x,D)\varphi \|_{\LL^1} \leq c\| A_{1}(x,D)\varphi \|_{\LL^1}
$$
holds for all $\varphi \in \CC^{\infty}_{c}(\Rn , \mathbb{V})$ if and only if there exists 
$C \in \LL^{\infty}(\Rn , \Lscr (\mathbb{W}_{1},\mathbb{W}_{2}))$ with $\| C \|_{\LL^\infty}\leq c$ such that
$$
a^{2}_{\alpha}(x)=C(x)a^{1}_{\alpha}(x) 
$$
for almost all $x$ and each multi--index $\alpha$ of length $k$.
\end{theorem}
We derive this result from a more general nonlinear version stated in Theorem \ref{nonlin}. This result is
in turn a consequence of Theorem \ref{convex} and standard arguments from the calculus of variations.
The case of constant coefficient homogeneous partial differential operators was given in \cite{KK11}.

The link between an Ornstein type result, concerning the failure of the $\LL^1$--version of Korn's 
inequality, and semiconvexity properties of the associated integrand -- though expressed in a dual 
formulation -- was observed already in  \cite{CFM05a}. There it was utilized in an ad-hoc 
construction which required a very sophisticated refinement in \cite{CFMM05b}, where it was 
transferred from an essentially two--dimensional situation into three dimensions.
Our arguments handle these situations with ease, see Corollary \ref{hessian} below. 

It is well--known that the distributional Hessian of a convex function $f\colon \Rn \to \R$ is a matrix--valued Radon
measure (see for instance \cite{reshet}). The natural question arises if this is 
valid also for the semi-convexity notions important in the vectorial calculus of variations. In \cite{CFMM05b} a 
fairly complicated construction was introduced to show that this is not true for
rank one convex functions defined on symmetric $2 \times 2$ matrices.
Here we wish to address the question of regularity of second derivatives when the function 
$f \colon \mV \to \R$ is merely $\mD$--convex. First note that when $f$ is $\CC^2$, then $\mD$--convexity is equivalent to 
$D^{2}f(x)[e,e] \geq 0$ for all $x \in \mV$ and $e \in \mD$. Hence using mollification and that a positive distribution is 
a Radon measure it follows that the Hessian of $\mD$--convex functions $f\colon \mV \to \R$ will be a Radon measure provided 
we can select a basis $(e_{j})$ for $\mV$ consisting of vectors from $\mD$ such that additionally $e_{i}\pm e_{j} \in \mD$ for 
all $1 \leq i < j \leq \mathrm{dim}\mV$. Indeed under this assumption on $\mD$ we can use polarisation to see that all second 
order partial derivatives (in the coordinate system defined by $(e_{j})$) are  measures. It is not difficult to show that the 
above condition on the cone of directions $\mD$ is equivalent to the statement that the only functions $f\colon \mV \to \R$ with 
the property that both $\pm f$ are $\mD$--convex, called $\mD$--affine functions, are the affine functions. This
in turn is equivalent to saying that the cone of $\mD$--convex quadratic forms on $\mV$,
$$
\q_{\mD} = \bigl\{ q : \, q \mbox{ is a $\mD$--convex quadratic form on } \mV \bigr\}
$$
is line--free. It is remarkable, but in accordance with our theme on Ornstein $\LL^1$ non--inequalities, that if the
cone of directions $\mD$ satisfies the technical condition
\begin{equation}\label{mysterious}
\exists \, \ell \in \mV^{\ast} \, \mbox{ such that } \, \overline{\mD} \cap \mathrm{ker} \ell = \{ 0 \} ,
\end{equation}
then this simple sufficient condition is also necessary.

\begin{theorem}\label{ghessian}
Let $\mV$ be a normed finite dimensional real vector space and $\mD$ be a cone of directions in $\mV$ satisfying the condition (\ref{mysterious}).
Then the distributional Hessian of $\mD$--convex functions $f\colon \mV \to \R$ is a $\odot^{2}\mV$--valued Radon
measure on $\mV$ if and only if the only $\mD$--affine functions are the affine functions.
Moreover, when there exists a nontrivial $\mD$--affine function, then there also exists a $\mD$--convex $\CC^1$ function $f\colon \mV \to \R$ 
with a locally H\"{o}lder continuous
\footnote{$\forall K \Subset \mV$ $\forall \alpha <1$ $\exists c>0$ \, $\forall$ $x$, $y \in K$: \, $\| Df(x)- Df(y) \| \leq c|x-y|^{\alpha}$} 
first differential $Df \colon \mV \to \mV^\ast$, but for which the distributional Hessian fails to be
a measure in {\em any} open nonempty subset $O$ of $\mV$, in the sense that for some unit vector $e \in \mV$ we have
$$
\sup \left\{ \int_{O} \! f(x) D^{2}\varphi (x)(e,e) \, \dd x \, : \, \varphi \in \CC^{\infty}_{c}(O) \mbox{ and } 
\sup \abs{\varphi} \leq 1 \right\} = \infty .
$$ 
\end{theorem}
We remark that the condition on $\ell \in \mV^{\ast}$ given in (\ref{mysterious}) is equivalent to stating that the form
$\ell \otimes \ell$ is an interior point of the cone $\q_{\mD}$. Let us record here some examples of cones of directions for the case of
square matrices $\mV=\R^{n \times n}$ that satisfy condition (\ref{mysterious}):
$$
\mD (\xi_{0},\varepsilon_{0}) = \biggl\{ a \otimes b : \, a , \, b \in \Rn \mbox{ and } \abs{a \cdot \xi_{0}b} 
\geq \varepsilon_{0}\abs{a}\abs{b} \biggr\} ,
$$
where $\xi_{0} \in \R^{n \times n}$ and $\varepsilon_{0}>0$ are fixed and chosen so $\mD (\xi_{0},\varepsilon_{0})$ spans $\R^{n \times n}$. 
Indeed, with the usual identifications $(\R^{n \times n})^{\ast} \cong \R^{n \times n}$,
for such cones we have that $\xi_{0} \otimes \xi_{0}$ is an interior point of $\q_{\mD (\xi_{0},\varepsilon_{0})}$. 
Since any $2\times 2$ minor of $n \times n$ matrices defines a nontrivial rank--one affine function on $\R^{n \times n}$ 
we infer in particular the existence of $\mD (\xi_{0},\varepsilon_{0})$--convex $\CC^{1,\alpha}$ functions on $\R^{n\times n}$ whose
Hessians are nowhere a measure. When we restrict attention to functions defined on {\em symmetric} real $n \times n$ matrices 
$\R^{n \times n}_{\rm sym}$ the situation becomes more satisfying since there the rank--one cone 
$\mD^{n}_{\rm sym}=\{ t \cdot a \otimes a : \, t \in \R \mbox{ and }a \in \Rn \}$ satisfies the condition (\ref{mysterious}). Indeed 
$\mathrm{Id} \otimes \mathrm{Id}$ is an interior point in $\q_{\mD^{n}_{\rm sym}}$. We state this result separately:

\begin{corollary}\label{hessian}
Let $n \geq 2$ be an integer. Then there exists a rank--one convex $\CC^1$ function $f\colon \R^{n \times n}_{\rm sym} \to \R$ with a locally
H\"{o}lder continuous first derivative, but whose distributional Hessian $D^{2}f$ is not a bounded measure in any open nonempty subset 
$O$ of $\R^{n \times n}_{\rm sym}$ in the sense that for some $e \in \R^{n \times n}_{\rm sym}$ we have
$$
\sup \left\{ \int_{O} \! f(x) D^{2}\varphi (x)(e,e) \, \dd x \, : \, \varphi \in \CC^{\infty}_{c}(O) \mbox{ and } 
\sup \abs{\varphi} \leq 1 \right\} = \infty .
$$
In particular, the function $f$ cannot be locally polyconvex at {\em any} $\xi \in \R^{n \times n}_{\rm sym}$.
\end{corollary}
Corollary \ref{hessian} was announced in \cite{KK11}. 

Finally, we remark that due to concentration effects on rank--$1$ matrices, see \cite{Al93}, our results also allow us to simplify
the characterization of BV Gradient--Young measures given in \cite{KR10b}, see Theorem \ref{cgym}. In fact, this was the original motivation 
for the work undertaken in the present paper. 

We end this introduction by giving a brief outline of the organization of the paper. In Section 2 we discuss and introduce our notation
and terminology, and also derive some preliminary results. The proofs of Theorem \ref{convex} and Corollary \ref{homorc1} follow in
Section 3. Section 4 is concerned with higher order derivatives of maps between finite dimensional spaces and relaxation of variational
integrals defined on them. Previous works in the calculus of variations on relaxation of multiple integrals depending on higher order derivatives
has mostly been based on {\em periodic test maps} whereas we use {\em compactly supported tests maps}. The two approaches are essentially
equivalent, and here we supply the details in the latter case culminating in Theorem \ref{relaxform}. The proofs of Theorems \ref{orne}, \ref{nonlin}
and \ref{ghessian} are presented in Section 5. Section 6 contains a discussion of BV gradient Young measures, and a characterization, stated
as Theorem \ref{cgym}, of these by duality with a certain subclass of quasiconvex functions.


\section{Preliminaries}\label{prelims}

This section fixes the notation, collects standard definitions and recalls some preliminary results that are all
more or less well--known.

We have attempted to use standard or at the least self--explanatory notation and terminology. Thus our notation 
follows closely that of \cite{rock} for convex analysis, \cite{Dac89} and \cite{Mu} for calculus of variations, and 
\cite{AFP} for Sobolev functions and functions of bounded variations. We refer the reader to these sources for further
background if necessary.

Let $\mV$ be a real vector space. A subset $C$ of $\mV$ is convex if it is empty or if for all points $x$ and $y$ in $C$ also
the segment $(x,y)$ between them is contained in $C$.
Let $f\colon \mV \to \bar{\R} := [-\infty , \infty ]$ be an extended real--valued function.
Then its effective domain is the set $\dom (f) := \{ x \in \mV : \, f(x) < \infty \}$ and its epigraph is the set
$\epi (f) := \{ (x,t) \in \mV \times \R : \, t \geq f(x) \}$. The function $f$ is convex if its epigraph is a convex subset
of the vector space $\mV \oplus \R$.

By a \textbf{cone of directions} $\mD$ in $\mV$ we understand throughout the paper a balanced and spanning cone in $\mV$: so 
$\mD \subseteq \mV$ has the properties that for all $x \in \mD$, $t \in \R$ we have $tx \in \mD$, and its 
linear hull equals $\mV$.

\begin{definition}\label{defDcon}
Let $\mV$ be a finite dimensional real vector space and $\mD$ a cone of directions.

\noindent
(i) A subset $\mA$ of $\mV$ is $\mD$--convex provided for any two points $x$, $y \in \mA$ with $x-y \in \mD$ also the
segment $[x,y] = \{ (1-\lambda )x+\lambda y : \, \lambda \in [0,1] \}$ is contained in $\mA$.

\noindent
(ii) For an arbitrary subset $\mS$ of $\mV$ and an extended real--valued function $f \colon \mS \to \bar{\R}$
we say that $f$ is \textbf{weakly $\mD$--convex} provided for any $x$, $y \in \mS$ such that $[x,y] \subset \mS$ and $x-y \in \mD$
the restriction $f|_{[x,y]}$ is convex. We say that $f$ is \textbf{$\mD$--convex} if the function 
$$
F(x) = \left\{
\begin{array}{ll}
f(x) & \mbox{ if } x \in \mS\\
\infty & \mbox{ if } x \in \mV \setminus \mS
\end{array}
\right.
$$
is a weakly $\mD$--convex function $F \colon \mV \to \bar{\R}$.
\end{definition}
The reader will notice that this terminology is inspired by \cite{BES63}, and that on $\mD$--convex subsets $\mS$ of $\mV$ there
is no difference between weak $\mD$--convexity and $\mD$--convexity of {\em real--valued} functions $f\colon \mS \to \R$.

\begin{lemma}\label{finiterc}
Let $\mV$ be a normed finite dimensional real vector space and $\mD$ a cone of directions in $\mV$. Assume 
$f \colon \mS \to [-\infty , \infty )$ is a $\mD$--convex function defined on an arbitrary subset $\mS$ of $\mV$. 
If $\mA$ is a connected component of the interior of $\mS$, then either $f \equiv -\infty$ on $\mA$ or $f > -\infty$ on $\mA$.
\end{lemma}

\begin{proof}
The proof relies on two standard observations. First, if $h\colon (a,b) \subset \R \to [-\infty ,\infty ]$
is a convex function such that $h(t)<\infty$ for all $t \in (a,b)$ and $h(t_{0})>-\infty$ for some $t_{0} \in (a,b)$, then
$h(t) >-\infty$ for all $t \in (a,b)$. This follows easily from the definition of convexity.
The second observation is that, since $\mD$ spans $\mV$, any two points of $\mA$
can be connected by a piecewise linear path in $\mA$ whose segments are all in directions from $\mD$.
This follows by choosing a basis for $\mV$ consisting of vectors from $\mD$, then declaring it to be an orthonormal 
basis for $\mV$, so that the above is a well--known property of connected open sets in euclidean spaces.

Now if we have $x \in \mA$ with $f(x) > -\infty$ and $\ell$ is a line through $x$ in a direction from 
$\mD$, then by convexity of $f|_{\mA\cap \ell}$ we infer that $f(y) >-\infty$ for all $y$ belonging to the
connected component of $\mA\cap \ell$ containing $x$. The result follows since we can connect
any point $y \in \mA$ with $x$ by such line segments contained in $\mA$.
\end{proof}

The following result and its proof is a variant of a result from \cite{BKK00} (compare Lemma 2.2 there) that in turn is a slightly more
precise version of a well--known estimate, see \cite{Dac89} Theorem 2.31, but 
here being adapted to the case of a general cone $\mD$ of directions.  

\begin{lemma}\label{lipschitz}
Let $\mV$, $\| \cdot \|$ be a normed finite dimensional real vector space and $\mD$
a balanced cone whose linear span is $\mV$. 
If $f\colon B_{2r}(x_0 ) \to \R$ is  $\mD$-convex, then
$f$ is locally Lipschitz in $B_{2r}(x_{0})$. More precisely  we have
\begin{equation}\label{eqlip}
\abs{f(x)-f(y)} \leq L \| x-y \|
\end{equation}
for all $x$, $y \in B_{r}(x_0 )$, where 
$$
L=\frac{c_{0}}{r}\mathrm{osc}(f , B_{2r}(x_0 )) 
$$ 
and the constant $c_{0}$ depends only on the norm $\| \cdot \|$ and the cone $\mD$.
\end{lemma}

\begin{proof}
For the entire proof we fix unit vectors $e_{1}, \, \dots \, , \, e_{n} \in \mD$  which form 
a basis of $\mV$, and note that due to the equivalence of all norms on $\R^n$ there
is a constant $c \in (1,\infty)$ such that 
\begin{equation}\label{eqnorm}
\frac{1}{c} \sum_{j=1}^n |t_j|\leq \| \sum_{j=1}^n t_j e_j \|  \leq c \sum_{j=1}^n |t_j| 
\end{equation}
for all $t_j\in \R$.
Now we proceed in three steps.

\noindent
{\it Step 1.}  The function $f$ is locally bounded.

We start by showing
that for $x \in B_{2r}(x_{0})$ and positive $\varepsilon $ such that the parallelepiped $C$ with center $x$ and sides
parallel to the basis vectors and all of length $2\varepsilon$ is contained in $B_{2r}(x_{0})$,
$$
C = \left\{ x+\sum_{j=1}^{n} t_{j}e_{j} : \, t_{1}, \, \dots \, , \, t_{n} \in [-\varepsilon , \varepsilon ] \right\} 
\subset B_{2r}(x_{0}),
$$
we have $\sup f(C) = \max f(K)$, where the set $K = \mathrm{ext}(C)$ consists of the vertices of 
the parallelepiped $C$.
Indeed, by convexity in each of the directions $e_j\in \mD$ we find recursively for each $x + \sum_{1}^{n}t_{j}e_{j} \in C$,
\begin{eqnarray*}
f(x+\sum_{j=1}^{n}t_{j}e_{j}) &\leq& \max_{\eps_{1}=\pm 1} f(x+\sum_{j=2}^{n}t_{j}e_{j}+\eps_{1}e_{1})\\
&\leq& \quad \dots\\
&\leq& \max f(K).
\end{eqnarray*}
Consequently, $\sup f(C) \leq \max f(K)$. In order also to get a lower bound we fix $y \in C$, say $y=x+\sum_{1}^{n}t_{j}e_{j}$
where each $|t_{j}| \leq \eps$. Put $K_{y} = \{ x+\sum_{1}^{n} \eps_{j}t_{j}e_{j} : \, \eps_{1}, \, \dots \, , \, \eps_{n} \in \{ -1,1 \} \}$.
Observe that $\sum_{z \in K_{y}} 2^{-n}z = x$ and $K_y$ is contained in $C$, so again using convexity in each of the directions $e_j$ we find
$$
f(x) \leq \sum_{z \in K_y} 2^{-n}f(z) \leq 2^{-n}f(y)+(1-2^{-n})\sup f(C),
$$
hence $f(y) \geq 2^{n}f(x)-(2^{n}-1)\max f(K)$, and therefore the lower bound 
$$
\inf f(C) \geq 2^{n}f(x)-(2^{n}-1) \max f(K)
$$
follows.

\noindent
{\it Step 2.} For all $x$, $y \in B_{r}(x_{0})$ with $x-y \in \mD$ we have $|f(x)-f(y)| \leq \tfrac{m}{r}\| x-y \|$, where
$m= \mathrm{osc}(f, B_{2r}(x_{0}))$. 

To prove this we can assume that $m < \infty$. Now fix $s \in (r,2r)$ and consider the line $L$ through $x$ and $y$. 
Select $z \in L \cap \partial B_{s}(x_{0})$ such that $x \in [y,z]$, the segment between $y$ and $z$. By assumption the 
restriction of $f$ to $L$ is convex and since $\| y-z \| \geq s-r$ we get
$$
\frac{f(x)-f(y)}{\| x-y \|} \leq \frac{f(z)-f(y)}{\| z-y \|} \leq \frac{m}{s-r}.
$$
The required conclusion follows since $s<2r$ was arbitrary and $x$ versus $y$ can be
swapped.

\noindent
{\it Step 3.}
Now consider any $x\in B_r(x_0)$ and $B_\eps(x)\subset B_r(x_0)$. If $\|x-y\|<\eps/c^2$ for
the constant $c$ from (\ref{eqnorm}), then 
$$y-x= \sum_{j=1}^n t_je_j \text{ where } \sum_{j=1}^n |t_j|< \frac{\eps}{c}.$$
Hence, the points
$$y_k=x+\sum_{j=1}^k t_j e_j, \quad k=1,\ldots,n$$
all belong to $B_{\eps}(x)$, and so in particular to $B_{r}(x_{0})$, since, again due to (\ref{eqnorm}),  
$$
\|y_k-x\|=\|\sum_{j=1}^k t_j e_j\|\leq c \sum_{j=1}^k |t_j|  <\eps .
$$
Consequently, step 2 and  (\ref{eqnorm}) ensure that 
$$|f(y)-f(x)|
=|f(y_n)-f(y_0)|\leq \sum_{j=1}^n |f(y_j)-f(y_{j-1})| 
\leq \sum_{j=1}^n \frac{m}{r} |t_j|\leq \frac{m}{r} c\|x-y\|.
$$
This shows that $f$ is everywhere inside $B_r(x_0)$ locally Lipschitz with a constant at most
$mc/r$. Because $B_r(x_0)$ is convex, the Lipschitz constant of $f$ on the entire ball can 
also not exceed $cm/r$.
 
Finally, we  observe that this argument also shows that $f$ is locally Lipschitz everywhere in  $B_{2r}(x_{0})$.
\end{proof}


An open convex cone $\mC$ in $\mV$ is an open subset $\mC$ of $\mV$ such that $x+ty \in \mC$ whenever $x$, $y \in \mC$ and $t>0$.

For a function $f\colon \mC \to \R$ defined on an open convex cone $\mC$ in $\mV$ we define the \textbf{(upper) recession function}
$f^{\infty}\colon \mC \to \bar{\R}$ as
\begin{equation}\label{urecess}
f^{\infty}(x) := \limsup_{\stackrel{t \to \infty , x^{\prime} \to x}{tx^{\prime} \in \mC}} \frac{f(tx^{\prime})}{t} \quad \quad (x \in \mC )
\end{equation}
The definition implies that $f^{\infty}$ is positively $1$-homogeneous: $f^{\infty}(tx) = tf^{\infty}(x)$ for all $x \in \mC$
and $t>0$. It also follows that when $f$ has linear growth on $\mC$, meaning that for some constant $c$ we have the bound
$$
\abs{f(x)} \leq c\bigl( \| x \| + 1 \bigr)
$$
for all $x \in \mC$, then $f^{\infty}$ is real--valued and $\abs{f^{\infty}(x)} \leq c\| x \|$ for all $x \in \mC^{\infty}$.
When $f$ is globally Lipschitz on $\mC$ with Lipschitz constant $\mathrm{lip}(f,\mC)=L$:
$$
\abs{f(x)-f(y)} \leq L \| x-y \| 
$$
for all $x$, $y \in \mC$, then we can simplify the definition of the recession function $f^{\infty}$ to
\begin{equation}\label{liprecess}
f^{\infty}(x) = \limsup_{t \to \infty} \frac{f(tx)}{t} \quad \quad (x \in \mC )
\end{equation}
and of course again $\mathrm{lip}(f^{\infty},\mC )\leq L$. Let us also remark that since all norms are equivalent on
a finite dimensional vector space and the actual Lipschitz constants play no role here we shall often just state that the function
under consideration has linear growth or is Lipschitz without specifying a norm. 

There is also a natural notion of lower recession function, but for
our purposes it is less useful. The reason is that the following result fails for lower recession functions.

\begin{lemma}\label{convrecess}
Let $\mC$ be an open convex cone in $\mV$ and assume that $f \colon \mC \to \R$ is a $\mD$--convex function of linear growth on $\mC$.
Then the recession function $f^{\infty}\colon \mC \to \R$ is again $\mD$--convex.
\end{lemma}

We omit the straightforward proof. The following simple observation turns out to be 
crucial in the sequel. 

\begin{lemma}\label{keybound} 
Let $\mV$ be a normed finite dimensional real vector space and $\mD$ a balanced cone of directions in $\mV$. 
Assume $\mC$ is an open convex cone in $\mV$ and that $f \colon \mC \to \R$ is a $\mD$--convex function. 
Then
\begin{equation}\label{keyboundeq}
f(x+y) \leq f^{\infty}(x)+f(y)
\end{equation}
for all $y \in \mC$ and $x \in \mC \cap \mD$. 
\end{lemma}

\begin{proof}
Let $y \in \mC$ and $x \in \mC \cap \mD$. By definition of open convex cone, $y+tx \in \mC$ for all $t \geq 0$, and
because $f$ is convex in the direction of $x$ we get for all $t \geq 1$,
$$
f(x+y)-f(y) \leq \frac{f(y+tx)-f(y)}{t}= \frac{f(t(x+\frac{y}{t}))}{t}-\frac{f(y)}{t},
$$
and in view of (\ref{urecess}) the conclusion follows by sending $t$ to infinity.
\end{proof}



\section{Proof of Theorem \ref{convex} } \label{easyfinish}

Throughout this section we assume that $\mV$ is a normed finite dimensional real vector space with
a cone of directions $\mD$.

\begin{definition} \label{defscone}
Let $\mC$ be a subset of $\mV$ and $f\colon \mC\to \R$ a function.
We say that $f$ has at $x \in \mC$
\begin{itemize}
\item
a \textbf{subdifferential}  if there is a linear function $\ell\colon \mV\to \R$
such that 
$$ f(y)\geq f(x)+\ell(y-x) \quad \text{ for all } y\in \mC.$$
\item
a \textbf{$\mD$--subcone} if there is a $\mD$--convex and positively $1$--homogeneous function
$\ell\colon \mV\to \R$ such that 
$$ 
f(y)\geq f(x)+\ell(y-x) \quad \text{ for all } y\in \mC.
$$
\end{itemize}
\end{definition}
As we shall see momentarily the two conditions are in fact equivalent. However, first we establish
the existence of $\mD$--subcones:

\begin{lemma} \label{existscone}
Let $\mC$ be a convex open cone in $\mV$. 

If $f\colon\mC\to \R$ is $\mD$--convex and positively $1$--homogeneous, then for any 
$x\in \mC\cap \mD$, $y\in \mC$ and $\lambda\in (0,1)$ we have 
\begin{equation}\label{monot}
f(y)-f(x)\geq \frac{f(x+\lambda(y-x))-f(x)}{\lambda}.
\end{equation} 
In particular, $f$ has a $\mD$--subcone at $x$.
\end{lemma}

\begin{proof}
Since adding a linear function to $f$ does not effect the assumptions nor the validity of 
(\ref{monot}) or the existence of a $\mD$--subcone, we can suppose in the sequel that 
$f(x)=0$. By Lemma \ref{lipschitz} the function $f$ is lipschitz near $x$ and hence also $f^\infty(x)=0$. 
Therefore, using Lemma \ref{keybound} we conclude as required
\begin{align*}
f(y)-f(x) & = f(y)= f(x+(y-x))\\
&\geq f(\frac{x}{\lambda} + (y-x))=\frac{f({x} +\lambda (y-x))-f(x)}{\lambda}.
\end{align*}

To see that this implies the existence of a $\mD$--subcone at $x$ we choose an $\eps>0$ so 
$B_{2\eps}(x)\subset \mC$ and define for $s \geq 1$,
$$
g_s(y) := sf(x+\frac{y}{s}), \quad y \in B_{s\eps}(0).
$$
Clearly $g_{s}(0)=0$. By Lemma \ref{lipschitz} we have $\lip(g_s)=\lip(f,B_\eps(x))$ and by
(\ref{monot}) we get for $s \leq t$ the monotonicity property $g_{s}(y) \geq g_{t}(y)$ for all $y \in B_{s\eps}(0)$.
Hence for each $y\in \mV$ the limit
$$
g(y) := \lim_{s\to\infty} g_s(y)=\inf_{s>0} g_{s}(y)
$$
exists in $\R$, and defines a Lipschitz function $g \colon \mV \to \R$. As a pointwise limit of $\mD$--convex
functions $g$ is $\mD$--convex, and since for any $y\in \mV$ and $\lambda >0$,
$$ 
g(\lambda y)=\lim_{s\to \infty} s f\left(x+\frac{\lambda y}{s}\right)=
 \lambda  \lim_{s\to \infty} \left(\frac{s}{\lambda}\right)  f\left(x+\frac{y}{s/\lambda}\right)=
\lambda g(y),
$$
$g$ is also positively $1$--homogeneous. Finally, for $y \in \mC$ we get from (\ref{monot}) upon taking
any $\lambda \geq \eps/(1+\| y-x \|)$ that $\lambda (y-x) \in B_{\eps}(0)$ so
$$
f(y)-f(x)\geq \frac{1}{\lambda}f\left(x+\lambda(y-x)\right)=
\frac{1}{\lambda} g_1(\lambda(y-x)).
$$
Since $g_1 \geq g$ and $g$ is positively one--homogeneous it follows that $g$ is a $\mD$--subcone for 
$f$ at $x$.
\end{proof}

\begin{proposition}\label{simple}
Let $\mC$ be an open subset of $\mV$. 
If the function $f\colon\mC\to \R$  has a $\mD$--subcone at the point $x_{0} \in \mC$, then
it also has a subdifferential at $x_{0}$.
\end{proposition}

\begin{proof} We will proceed by induction on the dimension $n\geq 1$ of $\mV$ and suppose that
the statement is true whenever the dimension of the vector space is less than $n$.

A moments reflection  on  Definition \ref{defscone} shows that existence both of a $\mD$--subcone 
and a subdifferential are unaffected by a simultaneous shift of the function $f$ and the 
contact point $x_0$, so we can without loss in generality suppose that $x_{0}=0_\mV$ and, replacing $f$ 
by its $\mD$--subcone, also that $f$ is $\mD$-convex and positively one-homogeneous on all of $\mV$. 

We  choose a basis $e_1,\ldots,e_n \in \mD$ of $\mV$ and put $x=e_1$ and  
$\tilde{\mV}=\text{span}\{e_2,\ldots, e_n\}$, so $\tilde{\mV}$ is spanned by $\tilde{\mD}=
\tilde{\mV}\cap \mD$.  According to Lemma \ref{existscone} there is a
$\mD$--subcone $g$  for $f$ in $x$, and clearly the restriction $g_{| \tilde{\mV}}$ is its own 
$\tilde{\mD}$--subcone at $0_{\tilde{\mV}}$. Therefore, by the induction assumption there is a 
subdifferential at the origin: a linear function $\tilde{l} \colon \tilde{\mV}\to\R$ such that 
$\tilde{\ell}(y)\leq g(y)$ whenever $y\in \tilde{\mV}$. In particular $\tilde{\ell}=0$ if 
$n=1$. 

Now we claim that
$$
\ell(tx+y)=tf(x)+\tilde{\ell}(y) \quad \text{ for } y\in \tilde{\mV} \text{ and } t\in \R
$$
is a subdifferential for $f$ at $0_\mV$. Once this claim is shown, our proof is finished.

For this purpose we first note that cleary $\ell$ is linear on $\mV$. Since $f(0)=0$, we need
only show that $f(z)\geq \ell(z)$ for each $z=tx+y$, $t\in \R$ and $y\in\tilde{\mV}$.
But if $t=1$ then we have for all $y\in \tilde{\mV}$ that 
$$\ell(x+y)=f(x)+\ell(y)\leq f(x)+g(y)\leq f(x+y),$$ 
according to the definition of a $\mD$--subcone $g$ of $f$ in $x$. 

By positive $1$--homogeneity of $f$  we get now for all $t>0, y\in\tilde{\mV}$ that
$$\ell(tx+y) =t\ell(x+\frac{y}{t})\leq t f(x+\frac{y}{t})=f(tx+y).$$
Finally, if $t\leq 0$, $y \in \tilde{\mV}$ we use Lemma \ref{lipschitz}  to infer
$f^\infty(x)=f(x)$ and now Lemma \ref{keybound} gives
$$
\ell(x+y)\leq f(x+y)\leq f^\infty((1-t)x) + f(tx+y)
$$
and so
$$
f(tx+y)\geq \ell(x+y)-(1-t)f(x)=\ell(tx+y),
$$
which finishes the proof.
\end{proof}

Clearly, Theorem \ref{convex} is now a direct consequence of Proposition \ref{simple}
and  Lemma \ref{existscone}.

\section{Higher order derivatives}\label{calculus}

We briefly recall some notation and concepts, mainly from multi--linear
algebra, that will  prove convenient for dealing with higher order 
derivatives. 

Starting from the standard $\Lscr (X,Y)=\{f\colon X\to Y\,;\, f \text{ linear}\}$ for 
given finite--dimensional real vector spaces $X,Y$, we set $\Lscr^0(X,Y)=Y$ and inductively 
$$
\Lscr^{k+1}(X,Y)=\Lscr (X,\Lscr^k(X,Y))
$$
for $k \in \N_0$. As usual (see \cite{fed69}, 1.9--1.10), we identify $\Lscr^k(X,Y)$ with $\Mscr^k(X,Y)$, the space of $Y$--valued $k$--linear 
maps on $X$. In the sequel, we are mainly interested in the space of all $Y$--valued symmetric $k$--linear maps
$$
\odot^k(X,Y)=\{\mu \in \Mscr^k(X;Y) \,: \, \mu (x_1,\ldots,x_k)= \mu (x_{\sigma(1)},\ldots,
x_{\sigma(k)}) \text{ for all }\sigma \in \text{Sym}(k)\},
$$
where $\text{Sym}(k)$ is the group of all permutations of the set 
$\{1,\ldots,k\}$. When $Y=\R$ we simply write $\odot^k X$ instead of $\odot^{k}(X, \R )$.

In terms of a basis $(e_j )_{j=1}^{n}$ for $X$, we can express a map $\xi \in \Mscr^k(X,Y)$  
as a $Y$--valued homogeneous polynomial of degree $k$ in $kn$ real variables by
{\em fully expanding all brackets}:
$$
\xi (x_1,\ldots,x_k)= \sum_{i_1=1}^n \ldots \sum_{i_k=1}^n \xi_{i_1,\ldots,i_k}
x_1^{i_1} \cdot \ldots \cdot x_k^{i_k},
$$
where 
$$
\xi_{i_1,\ldots,i_k}=\xi (e_{i_1},\ldots,e_{i_k}) \in Y \text{ and } x_j=\sum_{i=1}^n x_j^i e_i.
$$
In particular, observe that $\xi \in \Mscr^k(X,Y)$ is symmetric if and only if for one basis
 $( e_j )_{j=1}^{n}$ of $X$ (and then for all) $\xi_{i_1,\ldots,i_k}=
\xi_{i_{\sigma(1)}, \ldots, i_{\sigma(k)}}$ for all 
$\sigma \in \text{Sym}(k)$ and all indices $i_j\in \{1,\ldots,n\}$. Obviously, when $\mathrm{dim}Y=N$
we have $\mathrm{dim} \odot^{k}(X,Y) = N\binom{k+n-1}{k}$.

In view of the standard definition of the (Fr\'{e}chet--)derivative of a 
$\CC^k$ mapping $f\colon \Omega \subset X\to Y$ and Schwarz' theorem on interchangeability
of partial derivatives we have for all $x \in \Omega$ that $D^k f(x)\in \odot^k(X,Y)$.
 
The following calculations are of particular interest to us: If $\mu \in \odot^k(X,Y)$ and $f_\mu(x)=\mu (x,\ldots,x)$ 
for $x \in X$, then 
\begin{align*}
Df_\mu(x)(h) &= \lim_{t\to 0} \frac{1}{t} 
\biggl( \mu (x+th,\ldots,x+th) - \mu (x,\ldots,x) \biggr) \nonumber \\
&=\mu (h,x,x,\ldots,x)+ \mu (x,h,x,\ldots,x) + \ldots + \mu (x,\ldots,x,h) 
\nonumber \\
&=k \mu (h,x,\ldots,x). 
\end{align*}
Iterating this, we obtain 
$$
D^l f_\mu (x)(h_1,\ldots,h_l) = \frac{k!}{(k-l)!} 
\mu (h_1,h_2,\ldots,h_k,\underbrace {x,\ldots,x}_ {(k-l)-\text{times}}),
$$
and finally
\begin{equation}\label{eins}
D^k f_\mu (x)(h_1,\ldots,h_k) = k! \mu (h_1,h_2,\ldots,h_k)
\end{equation}   
for all $x$, $h_{1}, \, \dots \, , \, h_{k} \in X$.
In particular, the symmetric $k$--linear map generating a homogeneous 
$k$--th order polynomial $f$ is unique, the converse representation (of any 
such polynomial by a symmetric $k$--linear map) is a consequence of Taylor's Theorem. 
Thus, we see that $\odot^k(X,Y)$ is precisely the space of all $k$--th derivatives 
of $Y$--valued functions defined on (open subsets $\Omega$ of) $X$.  

Now we focus, for the sake of notational simplicity,  
on the only slightly more special situation of $X=\R^n$ with usual scalar product $\langle \cdot , \cdot \rangle$ or, 
equivalently, any finite dimensional inner product space. Analogous to before, we find that for all $\CC^k$ functions 
$g\colon \R \to \R$ and each $a\in \Rn$ the function 
$f(x)=g(\langle x,a\rangle)$ on $\Rn$ satisfies
$$
D^k f(x)(h_1, \ldots, h_k) = g^{(k)}(\langle x,a\rangle) \prod_{i=1}^k \langle a,
h_i\rangle,
$$
or, in more convenient tensor notation,
\begin{equation}
\label{zwei}
D^k f(x) = g^{(k)}(\langle x, a \rangle) a^{\otimes k}, 
\end{equation}
where we identify $a\in \Rn$ with $y \mapsto \langle y,a\rangle$ 
in $(\Rn)^*$ and accordingly
$$
a^{\otimes k}= \underbrace{a \otimes \ldots \otimes a}_k
$$
with the symmetric $k$--linear form 
$a^{\otimes k}(h_1 , \, \ldots \, , \, h_k ) = \prod_{i=1}^{k} \langle a,h_i \rangle$.
In particular we record that its coordinates with respect to any orthonormal basis
 $(e_j )_{j=1}^{n}$ for $\Rn$ are
$$
(a^{\otimes k})_{i_{1}\dots i_{k}} = \prod_{l=1}^{k}a_{i_l}, \quad \mbox{ where } 
\quad a_{i_l} = \langle a,e_{i_l} \rangle .
$$
Next, we assert that the balanced cone
$$
\mD_s (n,k) := \{t\cdot a^{\otimes k} \, : \, t\in \R \text{ and } a \in \Rn \}
$$
spans $\odot^k(\R^n)$. 
Indeed, otherwise we could find $\mu \in \odot^k(\arn)\setminus\{0\}$ which is perpendicular to all of 
$\mD_s (n,k)$  with respect to the canonical scalar product on $\odot \R^n$ extended from $\Rn$
$$
\langle \xi, \zeta \rangle =\sum_{i_1=1}^n \ldots \sum_{i_k=1}^n
\xi_{i_1\ldots i_k} \zeta_{i_1\ldots i_k},
$$
where $\xi_{i_{1} \ldots i_{k}},\zeta_{i_{1} \ldots i_{k}}$ are the coordinates with respect to the orthonormal basis
obtained from $(e_j )_{j=1}^{n}$. But observing that 
$$
\langle \mu ,a^{\otimes k} \rangle = 
\sum_{i_1=1}^n \ldots \sum_{i_k=1}^n
\mu_{i_1\ldots i_k} a_{i_1}\cdot \ldots \cdot a_{i_k}= \mu (a,\ldots,a),
$$
we have $f_\mu \equiv 0$ on $\arn$ and hence by (\ref{eins})
the contradiction $\mu=0$.  
As an easy consequence we see that, using standard notation, also
the balanced cone 
$$
\mD_s= \mD_s(n,k;Y)=
\{b\otimes \underbrace{a \otimes \ldots \otimes a}_k = b\otimes a^{\otimes k} 
\, ; \, a \in \arn \text{ and } b \in Y\}
$$ 
spans  $\odot^k(\arn ,Y)$. Similar observations 
can also be found in the Appendix A.1  of \cite{SanZap}.

Finally, we want to clarify why this cone 
$\mD_s(n,k ,Y)$ plays  
in the case of $k$--th order derivatives the role which the  usual cone of rank--one 
directions has for first order derivatives.
So let on a bounded smooth domain $\Omega$ of $\Rn$ an integrable map $f\colon \Omega \to Y$ be given and
assume that its distributional $k$--th derivative satisfies 
$$
D^k f(x) \in \{\xi , \eta \}\text{ a.e., where }\xi ,\eta \in \odot^k(\arn,Y).
$$ 
Then $g := D^{k-1}  f\colon \Omega \to \odot^{k-1}(\arn ,Y)$ is easily seen to belong to 
$\WW^{1,\infty}(\Omega , \odot^{k-1}(\Rn , Y)))$ and therefore is lipschitz with  $Dg \in \{ \xi,\eta \}$ a.e., where, 
by Rademacher's theorem, the derivative can be understood both distributionally and 
classically a.e. A by now standard argument (see \cite{BJ}) yields that
 $\xi$, $\eta$ as elements of $\Lscr (\arn,\odot^{k-1}(\arn,Y))$ satisfy $\text{rank}(\xi -\eta )\leq 1$. In
other words, when $\xi \neq \eta$ the kernel of $\xi -\eta$ must be $(n-1)$--dimensional so for a unit vector $\nu \in \R^{n}$ 
we have $(\xi -\eta )(x)= (\xi - \eta )(\langle x, \nu\rangle \nu)$ as elements of $\odot^{k-1}(\Rn ,Y)$ for all $x \in \Rn$. 
By the symmetry of $\xi -\eta$ we therefore get for all $x=x_1\in \Rn, x_2,\ldots ,x_k \in \Rn$
\begin{eqnarray*}
(\xi -\eta )(x_1,x_2,\ldots,x_k) &=& [(\xi -\eta )(x)](x_2,\ldots,x_k)=
(\xi -\eta )(\langle x_1,\nu \rangle \nu,x_2,\ldots, x_k )\\
&=&(\xi -\eta )(\langle x_1,\nu \rangle \nu,\langle x_2,\nu \rangle \nu,\ldots,
\langle x_k,\nu \rangle \nu)\\
&=& (\xi -\eta )(\nu,\nu,\ldots,\nu) \prod_{i=1}^k \langle x_i,\nu\rangle.
\end{eqnarray*} 
In other words, then necessarily $\xi -\eta  = [(\xi -\eta )(\nu,\nu,\ldots,\nu)] \otimes
\nu^{\otimes k}$. Conversely (\ref{zwei}) shows that for each 
$\mu \in \mD_s(n,k;Y)$ there is $f \in \CC^{k-1,1}(\arn,Y)$ such that
$D^k f \in \{0,\mu\}$ a.e., and so $D^k f$ is not essentially constant.

Next we present a result that corresponds to the gradient distribution changing 
techniques used in the theory of first order partial differential
inclusions. Since we are not forced to find exact solutions, as is usually done
in this theory, we avoid many technical difficulties and we can work exclusively within the class of
$\CC^\infty$--maps. First let us recall the Leibniz rule in several variables.
Let $\Omega$ be an open subset of $X$.
If $f\colon \Omega \to Y$ and $g\colon \Omega\to Z$ are $\CC^k$--maps
and $\Psi \colon Y\times Z\to V$ is bilinear, then for
each $x\in\Omega$   and $v=(v_1,\ldots, v_k)\in X^k$ we have 
\begin{equation}\label{leib}
D^k \biggl( y\mapsto \Psi(f(y), g(y)) \biggr) (x)(v)=\sum_{M\subset\{1,\ldots,k\}}
\Psi(\tilde{\nabla}^{M}f(x)(v) ,\tilde{\nabla}^{\{1,\ldots,k\}\setminus M}
g(x)(v)),
\end{equation}
where we denote $\tilde{\nabla}^M \varphi (x)(v) := D^l \varphi (x)(v_{i_1},\ldots,v_{i_l})$
if $M= \{i_1,\ldots,i_l\} \subset\{1,\ldots,k\}$ has cardinality $l$ (well--defined due to the 
symmetry of the higher order derivatives).
The rule can be easily shown by induction on $k$. 

\begin{proposition} \label{split}
Let $\Omega$ be a bounded open subset of $\R^n$, $k\in \N$ and 
$\xi$, $\eta \in \odot^k(\arn,Y)$ with $\xi -\eta \in \mD_s(n,k,Y)$ be given (we express this by saying that  
$\xi$, $\eta$ are {\bf rank-one connected}. Suppose   $\l \in (0,1)$ and let $u \colon \Omega\to Y$  satisfy
$D^k u(x)=C=\l \xi + (1-\l) \eta$ for all $x \in \Omega$ (so $u$ is in particular a polynomial of degree $k$). 
Then for each $\eps>0$ there are a compactly supported $\CC^\infty$ map $\varphi \colon \Omega \to Y$ and 
open subsets $\Omega_\xi,\Omega_\eta$ of \ $\Omega$ such that
\begin{itemize}
\item[(i)]  $D^k(u+\varphi)(x)=\xi$ if $x\in \Omega_\xi$,
               $D^k(u+\varphi)(x)=\eta $ if $x\in \Omega_\eta$ and \\
               $\text{\rm dist}(D^k(u+\varphi)(x),[\xi ,\eta ])<\eps$ for all 
               $x \in \Omega$, where $[\xi ,\eta ]=\{t \xi + (1-t)\eta \,:\, t\in [0,1]\}$.
\item[(ii)] $|\Omega_\xi |>(1-\eps)\l |\Omega |$ and 
               $|\Omega_\eta |>(1-\eps)(1-\l)|\Omega|$,
\item[(iii)]$| D^l \varphi (x)|<\eps$ for all $x \in \Omega$ and $l<k$.
\end{itemize}
\end{proposition}

\begin{proof} 
We start by choosing a $1$-periodic $\CC^\infty$-function $h\colon \R \to 
[\lambda-1,\lambda]$ with $\int_0^1 h=0$ and for which there
exist intervals $I_\xi ,I_\eta \subset (0,1)$ such that  $h=\l-1$ on $I_\xi$, $h=\l$
on $I_\eta$ and $|I_\xi |>(1-\eps/2)\l$, $|I_\eta |>(1-\eps/2) (1-\l)$. For instance,
we could use a (sufficiently 'fine') mollification of the $1$--periodic
extension of the function $(\l-1)\ONE_{(0,\l)}+\l \ONE_{(\l,1)}$. Now for each $l \in \{ 0, \ldots ,k \}$
we define recursively $\CC^\infty$ functions $h_l$ by 
$$
h_0=h \text{ and } h_l(0)=0, h'_l=h_{l-1} \text{ if }l\geq 1.
$$
Setting $H=h_k$, we notice that $h_1$ is again $1$--periodic and, by induction,
\begin{equation}\label{h:growth} 
|H^{(l)}(t)|=|h_{k-l}(t)|\leq \frac{\|h_1\|_\infty}{(k-l-1)!}|t|^{k-l-1} 
\end{equation}
for $0\leq l <k$ and $t\in \R$.
For later reference we also define $\hat{\chi}_\xi$  to be the $1$--periodic
function that equals $\ONE_{I_\xi}-|I_\xi |$ on $[0,1)$ and 
similarly for $\hat{\chi}_\eta$. Since these functions integrate to zero over
each interval of length $1$, we see that their 
indefinite integrals (and hence also distributional 
primitive functions) $F_\xi$, $F_\eta$ are bounded. 
Hence, partial integration gives for each smooth test function $\varphi \in \CC^{\infty}_{c}(\R )$ that
$$
\lim_{j \to \infty} \big| \int_\R \hat{\chi}_\xi (jt)\varphi(t)\, \dd t\big|
=\lim_{j \to \infty} \big| \int_\R F_\xi (jt) \frac{\varphi^{\prime}(t)}{j} \, \dd t \big|=0,
$$
and hence
\begin{equation}\label{osc:weak}
\tilde{\chi}_{\xi ,j} \tows | I_\xi | \quad \text{ and } \quad  
\tilde{\chi}_{\eta ,j} \tows | I_\eta | \quad \text{ in } \LL^{1}( \R )^* 
\end{equation}
where $\tilde{\chi}_{\xi ,j} (t)=\ONE_{(I_\xi +\mathbb{Z})}(jt)$ and 
$\tilde{\chi}_{\eta ,j} (t)=\ONE_{(I_\eta +\mathbb{Z})}(jt)$. 

Since $\xi$ and $\eta$ are rank--one connected, we find $b\in Y$ and 
$a\in \mathbb{S}^{n-1}$ such that $\eta -\xi =b\otimes a^{\otimes k}$. Without loss of generality, we can assume 
that $b\neq 0$. 
Fix a $\CC^\infty$ function $\psi \colon \arn \to [0,1]$ compactly 
supported in $\Omega$ such that the set $\Omega_+=\{x\in \Omega\,:\, \psi (x)=1\}$ satisfies 
\begin{equation}
\label{big:psi}
|\Omega_+|>(1-\eps/2) |\Omega |.
\end{equation}
Defining now 
$$
\varphi_j(x)=b\psi(x)  j^{-k} H(\langle x,ja \rangle), x \in \arn, \, j \in \N ,
$$
we claim that for $j\in \mathbb{N}$ sufficiently large the map $\varphi=
\varphi_j$ is the required mapping. 
The first, and crucial observation is that (\ref{zwei}) and (\ref{leib}) imply
that for $x \in \Omega$,
\begin{eqnarray*}
\| D^k \varphi_j (x)&-&b\otimes a^{\otimes k} \psi(x) H^{(k)}(\langle x,ja\rangle)\|\\
&& = \Big\| \sum_{l=1}^k \sum_{\card A\,=l} b \otimes 
\big(\tilde{\nabla}^A \psi(x) \otimes j^{-l}  H^{(k-l)} 
(j \langle x,a \rangle) a^{\otimes (\{1,\dots,k\}\setminus A)}\big)\Big\|  \\
&& \stackrel{(\ref{h:growth})}{\leq} \|b\|{ \sum_{l=1}^k}
c_{\psi,h_1,l}  j^{-l} (j|x|)^{l-1}
\leq \|b\|c'_{\psi,h_1,k,\diam(\Omega)} j^{-1} .
\end{eqnarray*}
Recalling that $H^{(k)}=h$ and that $\mu +b\otimes a^{\otimes k}\psi(x)h(\langle x,ja \rangle)\in [\xi ,\eta ]$ 
for all $x$ we see that the last inequality in (i) is 
satisfied provided $j>  \|b\|c'_{\psi,h_1,k,\Omega} /\eps$. 
Since $D \psi\equiv 0$ on
$\Omega_+$ we conclude for the very same reason  that the first two 
statements in (i) hold true for 
$$
\Omega_{\xi ,j}=\{x\in\Omega_+ \,:\, \langle jx,a \rangle \in I_\xi +\mathbb{Z}\}, 
$$ 
and
$$
\Omega_{\eta ,j}=\{x\in\Omega_+ \,:\, \langle jx,a \rangle \in I_\eta +\mathbb{Z}\}.
$$ 
Similarly, (\ref{leib}), (\ref{zwei}) and
(\ref{h:growth}) imply $\max_{x\in \Omega} |D^l\varphi_j (x)|\leq 
c''_{l,\psi,\text{diam}(\Omega)} j^{-1}$ if $l<k$ and so (iii) follows. 

Finally, to establish (ii) usually a disjoint decomposition of (a large part of)
$\Omega_+$ into cubes with a side parallel to $a$ is considered. However, for variety we prefer to give 
an analytic argument.  Denoting  
$s(t)=\mathcal{H}^{n-1}(\{x\in \Omega_+\,:\, \langle x,a\rangle=t\})$, we 
infer from Fubini's theorem and (\ref{osc:weak}) that
\begin{eqnarray*}
\lim_{j \to \infty} |\Omega_{\xi ,j}|=\lim_{j \to \infty}\int \! s(t) 
\tilde{\chi}_{\xi ,j}(t) \dd t &=& \int \! s(t) \dd t |I_\xi |=|I_\xi | |\Omega_+|\\
&>& (1-\frac{\eps}{2})^2 \l |\Omega|\\
&>& (1-\eps)\l|\Omega|,
\end{eqnarray*}
and similarly $\lim_{j \to \infty} |\Omega_{\eta ,j}|>(1-\eps)(1-\l)|\Omega|$.
So, choosing $j$ sufficiently large, also (ii) holds and the proof is finished.
\end{proof}

After these preparations we show how  functionals with the semiconvexity properties 
discussed in the previous sections naturally arise when relaxing energy functionals. This phenomenon is 
well known in the Calculus of Variations, but we could not find a concise treatment of the higher 
order derivative case that was based on compactly supported test maps in the literature. The preference
has been for periodic test maps, see for instance \cite{BCO81} and \cite{FM}, and while it is possible to
derive our results in that context too we found it useful and worthwhile to here record the approach 
based on compactly supported test maps. We therefore provide a
detailed exposition for the convenience of the reader. 

Recall that a subset $\mS$ of $\mV$ is $\mD$--convex if for all $x$, $y \in \mS$ such 
that $x-y \in \mD$ we have that the segment $[x,y]$ is contained in $\mS$.

\begin{theorem}\label{relaxform}
Let  $\mS$ be an open and $\mD_s(n,k;Y)$--convex subset of
$\mV=\odot^k(\arn,Y)$ and let the extended real--valued function $F\colon \mS\to [-\infty,\infty)$ be 
locally bounded above and Borel measurable.  

Then the relaxation of $F$ defined as 
$$
\F (\xi ) := \inf \left\{ \int_{(0,1)^{n}} \! F(\xi +D^{k}\varphi (x)) \, \dd x : 
\, 
\begin{array}{l}
\varphi \in \CC^{\infty}_{c} ((0,1)^{n} , Y)\\
\xi+D^{k}\varphi (x) \in \mS \mbox{ for all } x
\end{array} 
\right\}
$$
is a $\mD_s (n,k,Y)$--convex function $\F \colon \mS \to [-\infty , \infty )$. 

Moreover, for any strictly increasing function $\omega \colon [0,\infty ) \to [0,1]$ with
$\omega (0)=0$ and $t/\omega (t) \to 0$ as $t \searrow 0$ and any $\eps>0$ we can without changing the 
value of $\F$  restrict the infimum in its definiton to be taken only over  
those $\varphi$ which in addition satisfy $|D^l \varphi(x)|<\eps$ for all $x$ and $l<k$, and
$\| D^{k-1} \varphi (x)-D^{k-1} \varphi (y) \| \leq \varepsilon \omega ( |x-y| )$ for all $x$, $y$. 
\end{theorem}

\begin{proof} 
We start by checking that the infimum is unchanged if we restrict the test maps $\varphi$ as described.
Since we can replace $\omega$ by its concave envelope we can without loss of generality assume that
$\omega$ in addition is concave.
Hence fix $\omega$, $\eps$ as in the statement and let $\varphi \in \CC^{\infty}_{c} ((0,1)^{n} , Y)$ with 
$\xi+D^{k}\varphi (x) \in \mS$ for all $x$. We then take $j \in \N$ such that 
$$
2^{-j} \max_{0 \leq l < k} \| D^l \varphi \|_{\LL^{\infty}} <\eps \quad \mbox{ and } \quad 
4\| D^{k}\varphi \|_{\LL^{\infty}}\frac{\sqrt{n}2^{-j}}{\omega (\sqrt{n}2^{-j})} < \eps .
$$
Extend $\varphi$ to all of $\Rn$ by $(0,1)^n$ periodicity and define for $x \in (0,1)^n$,
$\psi (x) = 2^{-jk}\phi (2^{j}x)$. Then $\psi$ is an admissible test map for $\F$ and it is clear
that
$$
\int_{(0,1)^n} \! F(\xi + D^{k}\psi (x)) \dd x = \int_{(0,1)^n} \! F(\xi + D^{k}\varphi (x)) \dd x 
\quad \mbox{ and } \quad \max_{0 \leq l < k} \| D^l \psi \|_{\LL^{\infty}} <\eps .
$$
To check the $\omega$--H\"{o}lder continuity we let $x$, $y \in (0,1)^n$. Divide the cube $[0,1]^n$ into
$2^{jn}$ non--overlapping dyadic subcubes each of side length $2^{-j}$. If $x$, $y$ belong to the same dyadic 
subcube, say they both belong to $2^{-j}z+[0,2^{-j}]^{n}$, where $z \in \mathbb{Z}^n$, then clearly
\begin{eqnarray*}
\| D^{k-1}\psi (x) - D^{k-1}\psi (y) \| &=& 2^{-j}\| D^{k-1}\varphi (2^{j}x-z)-D^{k-1}\varphi (2^{j}y-z) \|\\
&\leq& \| D^{k}\varphi \|_{\LL^\infty}|x-y|\\
&\leq& \| D^{k}\varphi \|_{\LL^\infty}\frac{\sqrt{n}2^{-j}}{\omega (\sqrt{n}2^{-j})}\omega (|x-y|)\\
&\leq& \varepsilon \omega (|x-y|) .
\end{eqnarray*}
If $x$, $y$ belong to distinct dyadic subcubes, say $D_x$, $D_y$, then take $\bar{x} \in (x,y) \cap D_x$, 
$\bar{y} \in (x,y) \cap D_y$ and note that $|x-y| \geq |x-\bar{x}|+|y-\bar{y}|$ and $|x-\bar{x}|$, $|y-\bar{y}| \leq \sqrt{n}2^{-j}$.
Using that $\psi$ must vanish in small neighbourhoods of $\bar{x}$, $\bar{y}$ we estimate as before
\begin{eqnarray*}
\| D^{k-1}\psi (x) - D^{k-1}\psi (y) \| &\leq& 2^{-j} \| D^{k}\varphi \|_{\LL^\infty} \biggl( 2^{j}|x-\bar{x}| + 2^{j}|y-\bar{y}| \biggr)\\
&\leq& 2\| D^{k}\varphi \|_{\LL^\infty} \frac{\sqrt{n}2^{-j}}{\omega (\sqrt{n}2^{-j})}\biggl( \omega (|x-\bar{x}|) +\omega (|y-\bar{y}|) \biggr)\\
&\leq& \eps \omega ( |x-y| ) .
\end{eqnarray*}
Next we turn to the $\mD$--convexity. Because $F<\infty$ also $\F < \infty$ and it therefore suffices to show that 
$\F(\mu )\leq \l \F(\xi)+(1-\l)\F(\eta )$ if $\xi$, $\eta \in \mS$ with $\eta -\xi \in\mD$, $\l \in (0,1)$  and $\mu=\l \xi + (1-\l) \eta$.  
Fix real numbers $s$, $t$ so $s>\F(\xi )$ and $t>\F(\eta )$ and by the definition of relaxation find maps $\varphi_\xi$, 
$\varphi_\eta \in \CC^{\infty}_{c}((0,1)^{n},Y)$ such that $\xi+D^{k}\varphi_{\xi}(x) \in \mS$, $\eta+D^{k}\varphi_{\eta}(x) \in \mS$
for all $x$, and
\begin{equation}\label{sest}
\int_{(0,1)^{n}} \! F(\xi +D^{k}\varphi_\xi (x)) \dd x<s
\end{equation}
and
\begin{equation}\label{sedm}
\int_{(0,1)^{n}} \! F(\eta +D^{k}\varphi_\eta (x)) \dd x<t.
\end{equation}
Given $\delta >0$ we find $\eps\in (0,\delta )$ such that
$$
\sup\{ F(\zeta ) : \, \text{dist}( \zeta ,[\xi ,\eta ])<\eps\}\eps<\delta ,
$$
and $\zeta \in \mS$ if $\text{dist}(\zeta ,[A,B])<\eps$. Next,  
we set 
$u(x)= \mu (x,x,\ldots,x)/k!$ and choose $\phi$ satisfying the conditions (i), (ii) and (iii) of Proposition 
\ref{split}  for $\Omega=(0,1)^n$, in particular so that 
$\Omega_\xi =\text{int}\{ x \in \Omega : \, D^k(u+ \phi)(x) = \xi \}$ and $\Omega_\eta =
\text{int}\{ x \in \Omega : \, D^k (u+\phi)(x)=\eta \}$ satisfy 
$\Lcal^{n} \biggl(\Omega\setminus(\Omega_\xi \cup \Omega_\eta )\biggr)  < \eps \Lcal^{n} (\Omega )$.
Clearly, we find an $l\in \mathbb{N}$ such that the family $\q_\xi$ of all 
dyadic cubes of sidelength $2^{-l}$ entirely contained in $\Omega_\xi$ satisfies
$\Lcal^{n} (\bigcup \q_{\xi} ) > (1-\varepsilon )\Lcal^{n}(\Omega_\xi )$ and similarly
$\Lcal^{n}\biggl( \bigcup \q_\eta \biggr) >(1-\varepsilon )\Lcal^{n} (\Omega_\eta )$. Hence,
$$
\Omega_r=\Omega\setminus\big(\bigcup \q_\xi
\cup \bigcup \q_\eta \big)
$$
fulfills  $\Lcal^{n} (\Omega_r )<\eps$ and by Proposition \ref{split}(i), 
$\text{dist}(D^k (u+\phi)(x),[\xi ,\eta ])<\eps$ for  $x\in\Omega$. Therefore, 
$$
\int_{\Omega_r} \! F(D^k(u+\phi))(x) \dd x < \delta.
$$ 
Finally, we define similarly to the first part of the proof the function 
$$
\psi(x)=\begin{cases} (u+\phi)(x) &\text{if } x\in \Omega_r,\\
                                        (u+\phi)(x)+ 2^{-kl}\varphi_\xi (2^l (x-y)) 
						&\text{if } x \in y+[0,2^{-l}]^n\in \q_\xi ,\\
(u+\phi)(x)+ 2^{-kl}\varphi_\eta (2^l (x-y)) 
						&\text{if } x \in y+[0,2^{-l}]^n\in \q_\eta .
\end{cases}
$$
As $D^k \psi(x)=\xi +D^k \varphi_\xi (2^l(x-y))$ in the interior of 
$y+[0,2^{-l}]^n\subset \q_\xi$ we have for each such cube that 
$$
\int_{y+[0,2^{-l}]^n} \! F(D^k \psi(x))\dd x< s \Lcal^{n} \biggl( y+[0,2^{-l}]^n \biggr) ,
$$
see (\ref{sest}). Similarly, (\ref{sedm}) implies 
$$
\int_{y+[0,2^{-l}]^n} \! F(D^k \psi(x))\dd x< t \Lcal^{n} \biggl( y+[0,2^{-l}]^n \biggr) 
$$
for cubes  $y+[0,2^{-l}]^n \subset \q_\eta$.
Altogether, in view of Proposition \ref{split} (ii),
\begin{eqnarray*}
\int_{\Omega} \! F(D^k \psi) &\leq& \delta + s \Lcal^{n} \biggl( \bigcup \q_\xi \biggr) +t \Lcal^{n} \biggl( \bigcup \q_\eta \biggr)\\
&\leq&\delta + s \Lcal^{n} (\Omega_\xi ) + t\Lcal^{n} (\Omega_\eta ) \\
&\leq& \delta + s(1-\Lcal^{n}( \Omega_\eta )) + t (1-\Lcal^{n} (\Omega_\xi ))\\ 
&\leq& \delta + s(1-(1-\l)(1-\eps))+ t (1-\l(1-\eps))\\
&\leq& s\l+ t(1-\l) + (\delta + \eps(s(1-\l)+t\l ))\\
&\leq& s\l+ t(1-\l)+ \delta (1+s+t).
\end{eqnarray*}
It is also easy see that $\psi - u \in \CC^\infty_c ((0,1)^n,Y)$ because 
it is a finite sum of such functions (including $\phi$). 
Since $\delta>0$ was arbitrary, the proof is complete.
\end{proof}

\section{A generalization of Ornstein's $\LL^1$ Non--Inequality}\label{ornstein}

As was mentioned in the introduction Theorem \ref{convex} implies the Ornstein non--inequality in the
context of $x$--dependent vector valued operators, and even a version involving certain nonlinear $x$--dependent 
differential expressions: 

\begin{theorem}\label{nonlin}
Let $F \colon \Rn \times \odot^{k}(\Rn , \RN ) \to \R$ be a Carath\'{e}odory integrand
satisfying
$$
F(x,t\xi ) = |t|F(x,\xi ) 
$$
and
$$
|F(x,\xi )| \leq a(x)|\xi | 
$$
for almost all $x \in \Rn$, all $t \in \R$ and $\xi \in \odot^{k}(\Rn , \RN )$, where
$a \in \LL^{1}_{\rm loc}(\Rn )$ is a given function.
Then 
\begin{equation}\label{nonlinorn}
\int_{\Rn} \! F(x,D^{k} \phi (x)) \dd x \geq 0
\end{equation}
holds for all $\phi \in \CC^{\infty}_{c}(\Rn , \RN )$ if and only if $F(x,\xi ) \geq 0$ for almost 
all $x \in \Rn$ and all $\xi \in \odot^{k}(\Rn , \RN )$.
\end{theorem}

\noindent
Before presenting the short proof of Theorem \ref{nonlin} let us see how it implies the Ornstein $\LL^1$
non--inequality stated in Theorem \ref{orne}. 

\begin{proof}[Proof of Theorem \ref{orne}]
Only one direction requires proof. To that end we write the differential operators in terms of linear mappings 
$\tilde{A}_{i}(x) \colon \odot^{k}(\Rn , \mathbb{V}) \to \mathbb{W}_i$ as follows:
$$
\tilde{A}_{i}(x)D^{k}\varphi = \sum_{\abs{\alpha}=k} a^{i}_{\alpha}(x)\partial^{\alpha}\varphi
$$
for all $\varphi \in \CC^{\infty}_{c}(\Rn , \mathbb{V})$.  Define 
$$
F(x,\xi ) := c \|\tilde{A}_{1}(x)\xi  \|- \|\tilde{A}_{2}(x)\xi \| \quad \mbox{ for }\quad (x,\xi )\in \Rn \times \odot^{k}(\Rn ,\mathbb{V}), 
$$
and note that Theorem \ref{nonlin} applies to the function $F$. Accordingly, $F(x, \cdot )$ is a nonnegative function for 
almost all $x$, and therefore the set inclusion, $\mathrm{ker} \, \tilde{A}_{2}(x) \supset \mathrm{ker} \, 
\tilde{A}_{1}(x)$, must in particular hold for the kernels for all such $x$. Fix $x$ such that $F(x, \cdot )$ is nonnegative.
We define a linear mapping $C(x)\colon \mathbb{W}_1 \to \mathbb{W}_2$ by 
$$
C(x)=\tilde{A}_{2}(x) \circ 
\biggl( \tilde{A}_{1}(x)|_{\bigl( \mathrm{ker}\tilde{A}_{1}(x) \bigr)^{\perp}}\biggr)^{-1}
\circ\text{proj}_{\mathrm{im}\tilde{A}_{1}(x)} .
$$ 
Hereby $C$ is defined almost everywhere in $\Rn$, is measurable, and $\tilde{A}_{2}(x) = C(x)\tilde{A}_{1}(x)$ holds for almost all $x$. 
Finally, the nonnegativity of $F(x, \cdot )$ yields the uniform norm bound $\| C \|_{\LL^{\infty}} \leq c$.
\end{proof}

\begin{proof}[Proof of Theorem \ref{nonlin}]
We assume that the integral bound (\ref{nonlinorn}) holds and shall deduce the pointwise bound $F \geq 0$.
First we reduce to the autonomous case: $F=F(\xi )$. The procedure for doing this is a variant of a well--known
proof showing that quasiconvexity is a necessary condition for sequential weak lower semicontinuity
of variational integrals (see \cite{mor66} Theorem 4.4.2). Let us briefly comment on the details. Denote by $C = (-1/2,1/2)^n$ the open
cube centred at the origin with sides parallel to the coordinate axes and of unit length, and let $\phi \in \CC^{\infty}_{c}(C,\RN )$.
We extend $\phi$ to $\Rn$ by $C$ periodicity, denote this extension by $\phi$ again, and put for $x_0 \in \Rn$, $r>0$, $j \in \N$,
$$
\varphi_{j}(x) := \left( \frac{r}{j}\right)^{k}\phi \left( j\frac{x-x_{0}}{r} \right) , \, x \in \Rn .
$$
Writing $C(x_{0},r) = x_{0}+rC$ we have for all $j \in \N$ that $\varphi_{j}|_{C(x_{0},r)} \in \CC^{\infty}_{c}(C(x_{0},r),\RN )$, and
so extending this restriction to $\Rn \setminus C(x_{0},r)$ by $0 \in \RN$ to get a test map for (\ref{nonlinorn}) we arrive at
$$
0 \leq \int_{C(x_{0},r)} \! F \bigl( x,D^{k}\phi \left( j\frac{x-x_{0}}{r} \right) \bigr) \, \dd x.
$$
If we let $j$ tend to infinity, then by a routine argument that uses the Riemann--Lebesgue lemma (see for instance \cite{Dac89} Chapter 1)
it follows that
$$
0 \leq \int_{C(x_{0},r)} \! \int_{C} \! F(x,D^{k}\phi (y)) \, \dd y \, \dd x.
$$
Since $x_0$, $r$ were arbitrary we deduce by the regularity of Lebesgue measure that
\begin{equation}\label{rieleb}
\int_{C} \! F(x,D^{k}\phi (y)) \, \dd y \geq 0
\end{equation}
holds for all $x \in \Rn \setminus N_{\phi}$ where $N_\phi$ is a negligible set. To see that (\ref{rieleb}) in fact holds
outside a negligible set $N \subset \Rn$ that is independent of $\phi$ one invokes the separability of
the space $\CC_{c}(C, \odot^{k}(\Rn , \RN ) )$ in the supremum norm. We leave the precise details of this to the
interested reader and also remark that a straightforward scaling argument shows that we may replace $C$ by $\Rn$
and allow any $\phi \in \CC^{\infty}_{c}(\Rn , \RN )$ in (\ref{rieleb}). It now remains to prove the
pointwise bound in the autonomous case $F=F(\xi )$.
By Theorem \ref{relaxform} the relaxation $\F$ given by the formula
$$
\mathcal{F}(\xi ) = \inf_{\varphi \in \CC^{\infty}_{c}((0,1)^{n}, \RN )} \int_{(0,1)^n} \! F(\xi + D^{k}\varphi (x)) \dd x
$$
is a rank--one convex function $\F \colon \odot^{k}(\Rn , \R^N ) \to [-\infty ,\infty )$. Recall that by {\em rank--one convex} function in
this context we understand a function which is convex in the directions of the cone $\mD_{s}(n,k, \R^N )$. Now
the integral bound (\ref{nonlinorn}) translates to $\mathcal{F}(0) \geq 0$, and it then follows from Lemma \ref{finiterc} that
$\mathcal{F}$ is a rank--one convex function which is real--valued {\em everywhere}. By the homogeneity of $F$ and of $\CC^{\infty}_{c}$ one 
also easily checks that $\mathcal{F}$ is $1$--homogeneous. Consequently, as the cone $\mD_{s}(n,k; \R^N )$ is balanced and
spanning, Theorem \ref{convex} implies in particular that $\mathcal{F}$ is convex at
$0 \in \odot^{k}(\Rn , \RN )$: there exists $\zeta \in \odot^{k}(\Rn , \RN )$ such that
$\mathcal{F}(\xi ) \geq \langle \xi , \zeta \rangle$ for all $\xi \in \odot^{k}(\Rn , \RN )$.
Taking $\xi = \pm \zeta$ and using $1$--homogeneity we see that necessarily $\zeta =0$. The conclusion
follows because $F \geq \mathcal{F}$.
\end{proof}

\begin{remark}
The proof given above for Theorem \ref{nonlin} can easily be adapted to also prove the following
result. For a $\CC^{\infty}$ map $\phi \colon \Rn \to \RN$ we denote by $J^{k-1}\phi (x)$
its $k-1$ jet at $x$: the vector consisting of all partial derivatives $D^{\alpha}\phi_{i} (x)$
of coordinate functions $\phi_i$ corresponding to all multi--indices $\alpha$ of length at most $k-1$ 
arranged in some fixed order. Hereby $J^{k-1}\phi \colon \Rn \to \R^d$ for a $d=d(n,k,N)$.

Let $F \colon \Rn \times \R^d \times \odot^{k}(\Rn , \RN ) \to \R$ be a Carath\'{e}odory integrand
(measurable in $x \in \Rn$ and jointly continuous in $(y,\xi )$) satisfying
$$
F(x,y,t\xi ) = |t|F(x,y,\xi )
$$
and
$$
|F(x,y,\xi )| \leq a(x)(|y|+|\xi |) 
$$
for almost all $x \in \Rn$, all $y \in \R^d$, $t \in \R$ and $\xi \in \odot^{k}(\Rn , \RN )$, where
$a \in \LL^{1}_{\rm loc}(\Rn )$ is a given function.

Then 
$$
\int_{\Rn} \! F(x,J^{k-1}\phi (x), D^{k} \phi (x)) \dd x \geq 0
$$
holds for all $\phi \in \CC^{\infty}_{c}(\Rn , \RN )$ {\em only if}
$F(x,0,\xi ) \geq 0$ for almost all $x \in \Rn$ and all $\xi \in \odot^{k}(\Rn , \RN )$.
\end{remark}

We now turn to:

\begin{proof}[Proof of Theorem \ref{hessian}]
Only the necessity part requires a proof.
Assume that there exists a nontrivial $\mD$--affine quadratic form $m$ on $\mV$, where $\mD$ is a cone of directions
satisfying (\ref{mysterious}).

In the setting of Section \ref{calculus} we consider the spaces $Y=\R$ and $X = \mV$.


The second derivatives that we are interested in will belong to the space
$\odot^2 (\mV ) \cong \odot^{2}(\mV ,\R)$.  The corresponding {\em rank-one cone} is the usual one in $\odot^2 (\mV )$, namely 
$\{ t \cdot \ell \otimes \ell : \, t \in \R \mbox{ and } \ell \in \mV^{\ast} \}$ (indeed, if we choose a basis for $\mV$ we 
may identify $\mV \cong \Rn$, $\odot^2 (\mV ) \cong \R^{n \times n}_{\rm sym}$ and the rank-one cone is simply the cone of
symmetric $n \times n$ matrices of rank at most one). This cone is of course balanced and spanning.
Next, and more specific to our  situation we observe that for an open subset $O$ of $\mV$ a $\CC^2$ function 
$\varphi \colon O \to \R$ is $\mD$--convex if for all $x\in O$ we have
$D^2 \varphi(x) \in \mC$ where
$$
\mC := \biggl\{ \mu \in \odot^{2} (\mV ) \, : \; \mu (e,e) > 0 \text{ for all } e \in \mD\setminus \{ 0 \} \biggr\} .
$$
Clearly, $\mC$ is an open convex cone, and in the notation from the Introduction we have $\mC = \mathrm{int} \q_{\mD}$.
The assumption (\ref{mysterious}) therefore amounts to $\ell \otimes \ell \in \mC$ for some $\ell \in \mV^{\ast}$. Put
$\mu_{0} = \ell \otimes \ell$ so that $\mu_0$ is a rank one form in $\mC$. If $\mu_{m} = D^{2} m(0)$ we have that 
$\mu_{0}+t\mu_{m} \in \mC$ for all $t \in \R$. We now consider the function $F( \mu ) := -\| \mu \|$ on $\mC$. By Theorem 
\ref{relaxform} the relaxation $\F$ is rank-one convex, and since $F< \infty$ also $\F < \infty$. We assert that 
$\F \equiv -\infty$. Assume not, so that $\F$ is finite somewhere in $\mC$. Then by Lemma \ref{finiterc}
$\F > -\infty$ everywhere on $\mC$, so that $\F$ is a real--valued and rank--one convex function on $\mC$.
From the definition of $\F$ and the homogeneity of $F$, $\CC^{\infty}_{c}$ we infer that $\F$ is
positively $1$--homogeneous. By virtue of Theorem \ref{convex} $\F$ is therefore convex at $\mu_0$ so that in particular:
\begin{eqnarray*}
\F (\mu_0 ) & \leq & \frac{1}{2}\biggl( \F (\mu_{0}+t\mu_{m})+\F ( \mu_{0}-t\mu_{m})\biggr)\\
& \leq & -\frac{1}{2}\biggl( \| \mu_{0}+t\mu_{m} \| + \| \mu_{0}-t\mu_{m} \| \biggr) ,
\end{eqnarray*}
which is impossible for large values of $t$. This contradiction shows that $\F\equiv -\infty $ on $\mC$. 

In view of the definition of the relaxation $\F$ we then find for any $\mu \in \mC$ and each 
$\eps>0$ an $\varphi\in \CC_c^\infty(Q,\R )$ with $\mu+D^2 \varphi(x)\in \mC$
for all $x\in \mV$ such that $\int_Q (-\|\mu + D^2 \varphi\|)<-1/\eps$, 
$\| \varphi \|_{L^\infty(Q)}+\| D \varphi\|_{L^\infty(Q)}<\eps$, and 
$$
[ D\varphi ]_{\omega} := \sup \left\{ \frac{\| D\varphi (x)- D\varphi (y)\|}{\omega (|x-y|)} : 
\, x,y \in Q, \, x \neq y \right\} < \eps ,
$$ 
where $Q$ is the unit cube in $\mV$ with respect to some (from now on fixed) orthonormal basis of $\mV$ and the integral 
is calculated with respect to the corresponding Lebesgue measure. There is nothing special by the open unit cube $Q$
in the above, and it is possible to replace it by any bounded open subset $O$ of $\mV$. Indeed, we
simply replace $\varphi$ by $x\mapsto \l^{-2}\varphi(\l x+y)$ for $\l >0$ and $y\in \mV$ and express
$O$ as a disjoint union of closed cubes $y+\lambda \bar{Q}$ to achieve this.

The existence of the function claimed in the statement of Theorem \ref{hessian}
is now an easy consequence of the Baire category theorem. 
We consider the set $A$ of all rank--one convex $\CC^1$ functions $f \colon X\to \R$ such that 
$$
\|f\|=\sup \left\{\frac{|f(x)|}{1+\|x\|^2 } + \frac{\| Df(x)- Df(y)\|}{\omega (|x-y|)(1+|x|+|y|)}
\;:\; x, y \in X, \, x \neq y \right\} 
$$
is finite and equip it with the metric $d(f,g) := \| f-g \|$. It is easy to verify that $(A,d)$ is a complete metric space,
and we denote by $B$ the closure of smooth functions in $A$. With the inherited metric this is clearly also a complete metric
space. Now let $(O_l)_{l=1}^\infty$ be a sequence consisting of all dyadic cubes in $\mV$ of side length at most one and define, 
for $l,k\geq 1$,
$$
B_{l,k}=\left\{f \in B\;:\; \int_{O_l} f \frac{\partial^2 \psi}{\partial x^{\alpha}} \leq k \text{ if }
\psi \in \CC^\infty_c(O_l,\R), \|\psi\|_\infty\leq 1 \text{ and } \abs{\alpha} =2 \right\}.
$$
Clearly, each $B_{l,k}$ is a closed set in $B$ and for $f \in \CC^2$ the usual integration by parts argument
applied to a $\psi$ that sufficiently well approximates $\text{sign}({\partial^2 f/\partial x^{\alpha}})$ 
shows that we have 
$$
f\in B_{l,k}\quad\Leftrightarrow \quad
\int_{O_l}\left| \frac{\partial^2 f}{\partial x^{\alpha}}\right|\leq k \quad \mbox{ for all } \abs{\alpha}=2.
$$
We finish our proof by showing that each $B_{l,k}$ has empty interior since then  
$B\setminus \bigcup_{l,k} B_{l,k}$ is nonempty and obviously each $f$ in this set satisfies the conclusion of 
Theorem \ref{hessian}. But otherwise, we would find an open ball $B(f,\delta)$ inside some 
$B_{l,k}$. 
The definition of $B$ as the closure of smooth functions means that we without loss of 
generality may assume $f\in \CC^\infty$. Now we find a $C\in \mC$ considered above such that 
$\|\phi_C\|<\delta/3$, here $\phi_C(x)=\frac12 C(x,x)$ and the first part of our proof 
ensures the existence of a $\varphi\in \CC^\infty_c(O_l)$ such that 
$\|\varphi\|\leq \|\varphi\|_\infty<\delta/3$ but $\int_{O_l} \| D^2 (\varphi_C+\varphi )\|>3n^4 k$.
 In particular, there must exist a multi-index $\alpha$ of length two such that 
$\int_{O_l}  |\partial^2( \varphi_C+\varphi)/\partial x^{\alpha}|>3k$. From this it is clear that
$f+\varphi_C+\varphi \in B(f,\delta)$ but   
$$
\int_{O_l} \! \left|\frac{\partial^2( f+\varphi_C+\varphi)}{\partial x^{\alpha}}\right|>2k,
$$
a contradiction finishing our proof of Theorem \ref{hessian}. 

\end{proof}

\section{Gradient Young Measures} \label{applications}

A convenient way to describe the one--point statistics in a bounded sequence
of vector--valued measures where both oscillation and concentration phenomena must
be taken into account is through a notion of Young measure that was
introduced by DiPerna and Majda in \cite{DM87}. This formalism was revisited in \cite{AB97}
and a special case of particular importance was developed into a very convenient and suggestive
form. This was the starting point for \cite{KR10b}, and we shall briefly 
pause to describe the relevant set--up here.

Throughout this section $\Omega$ denotes a fixed bounded Lipschitz domain in $\Rn$ and $\Hi$ a
finite dimensional real inner product spaces.

Let $\Eb (\Omega , \Hi )$ denote the space of \textbf{test integrands} consisting of all continuous 
functions $f \colon \cl{\Omega} \times \Hi \to \R$ for which
$$
f^{\infty}(x,z) = \lim_{t \to \infty} \frac{f(x,tz)}{t}
$$
exists locally uniformly in $(x,z) \in \cl{\Omega} \times \Hi$. Note that hereby the
\term{recession function} $f^{\infty}$ is jointly continuous and $z \mapsto f^{\infty}(x,z)$
is for each fixed $x$ a positively $1$--homogeneous function.

The set of Young measures under consideration, denoted by $\Yb(\Omega , \Hi )$, is defined to be the set of 
all triples $( (\nu_{x} )_{x \in \Omega},\lambda , (\nu^\infty_{x} )_{x \in \overline{\Omega}})$ such that
\begin{align*}
\bullet  \quad &(\nu_x)_{x \in \Omega} \mbox{ is a weakly}^\ast \Lcal^{n} \mbox{ measurable family of probability measures on } \Hi ,\mbox{ and}\\
  & \int_{\Omega} \! \int_{\Hi} \! |z| \, \di \nu_x (z) \dd x < \infty ,\\
\bullet \quad &\lambda \mbox{ is a nonnegative finite Radon measure on } \cl{\Omega}, \\
\bullet \quad  &(\nu_x^\infty)_{x\in \cl{\Omega}} \mbox{ is a weakly}^\ast \lambda \mbox{ measurable family of probability measures on } 
\SH , \mbox{ where } \SH\\ 
  & \mbox{is the unit sphere in } \Hi . 
\end{align*}
Here, the maps $x \mapsto \nu_x$ and $x \mapsto \nu_x^\infty$ are only defined up to $\Lcal^n$- and 
$\lambda$--negligible sets, respectively. The parametrized measure $(\nu_x)_x$ is called the \term{oscillation measure}, 
the measure $\lambda$ is the \term{concentration measure} and $(\nu_x^\infty)_x$ is the \term{concentration-angle measure}. 
As the terms indicate, the oscillation measure is the usual Young measure and describes the oscillation (or failure of convergence in $\Lcal^n$ measure), 
the concentration measure describes the location in $\cl{\Omega}$ and the magnitude of concentration (or failure of $\Lcal^n$ equi--integrability) 
while the concentration--angle measure describes its direction in $\Hi$. Often, we will simply write $\nu$ as short-hand for the above triple.

To every $\Hi$--valued Radon measure $\gamma$ supported on $\cl{\Omega}$
with Lebesgue--Radon--Nikod\'{y}m decomposition $\gamma = a \Lcal^d + \gamma^s$, we 
associate an \textbf{elementary Young measure} $\epsilon_\gamma \in \Yb(\Omega , \Hi )$ given by the definitions
$$
  (\epsilon_\gamma)_x := \delta_{a(x)},  \qquad \lambda_{\epsilon_\gamma} := \abs{\gamma^s},
    \qquad (\epsilon_\gamma)_x^\infty := \delta_{p(x)},
$$
where $p := \di \gamma^{s}/\di \abs{\gamma^s}$ is the  Radon--Nikod\'{y}m derivative and $\abs{\gamma^s}(=\abs{\gamma}^{s})$ 
denotes the total variation measure of $\gamma^s$.

Using the inner product on $\Hi$ we can define a duality pairing
of $f \in \Eb (\Omega , \Hi )$ and $\nu \in \Yb (\Omega , \Hi )$ by setting
$$
\ddpr{f , \nu} := \int_{\Omega} \! \int_{\Hi} \! f(x,z) \dd \nu_{x}(z) \dd x + \int_{\cl{\Omega}} 
\! \int_{\SH} \! f^{\infty}(x,z) \dd \nu_{x}^{\infty}(z) \dd \lambda (x).
$$
For a sequence $(\gamma_j)$ of $\Hi$--valued Radon measures with $\sup_j \abs{\gamma_j}(\cl{\Omega}) < \infty$, we say 
that $( \gamma_j )$ \term{generates} the Young measure $\nu \in \Yb(\Omega , \Hi )$, in 
symbols $\gamma_j \toY \nu$, if for all $f \in \Eb(\Omega , \Hi )$ we have $\ddpr{f, \epsilon_{\gamma_{j}}} \to \ddpr{ f, \nu }$, or
equivalently,
\begin{equation} \label{eq:generation}
f \Big( x, \frac{\di \gamma_j}{\di \Lcal^n}(x)\Big) \,\Lcal^n
      + f^\infty \Big(x, \frac{\di \gamma^s_j}{\di \abs{\gamma^s_j}}(x) \Big) \, \abs{\gamma^s_j}
\to  \dprb{f(x,\cdot ), \nu_x} \, \Lcal^n  + \dprb{f^\infty(x,\cdot ),
      \nu_x^\infty} \, \lambda   
\end{equation}
weakly$\mbox{}^\ast$ in $\CC (\cl{\Omega})^{\ast}$ as $j \to \infty$.
A continuity result due to Reshetnyak \cite{R68} can be used to justify the formalism (see \cite{KR10b}) in the sense that
for measures $\gamma_{j}$, $\gamma$ not charging $\partial \Omega$ we have
$\gamma_{j} \to \gamma$ $\langle \cdot \rangle$--strictly if and only if $\epsilon_{\gamma_{j}} \toY \epsilon_{\gamma}$. Here 
the former is taken to mean that $\gamma_{j} \tows \gamma$ in $\CC_{0}(\Omega , \Hi )^{\ast}$ and 
$\langle \gamma_{j} \rangle (\Omega ) \to \langle \gamma \rangle (\Omega )$, where
$$
\langle \gamma \rangle (\Omega ) := \int_{\Omega} \! \left( 1 + \abs{\frac{\di \gamma}{\di \Lcal^{n}}}^{2} \right)^{\frac{1}{2}} 
\, \di \Lcal^{n} + \abs{\gamma^{s}}(\Omega ) .
$$
Standard compactness theorems for measures imply that any bounded sequence $(\gamma_j)$ of $\Hi$--valued Radon measures 
on $\overline{\Omega}$ admits a subsequence (not relabelled) such that $\gamma_j \toY \nu$ for some Young measure 
$\nu \in \Yb(\Omega , \Hi )$. We are particularly interested in the situation where the measures $\gamma_j$ are 
concentrated on $\Omega$.
Making special choices of test integrands $f$ in (\ref{eq:generation}) we deduce that
$\gamma_{j} \tows \gamma$ in $\CC_{0} (\Omega , \Hi )^{\ast}$, where $\gamma = \bar{\nu} \lfloor \Omega$ and
$$
\bar{\nu} = \bar{\nu}_{x}\Lcal^n +\bar{\nu}_{x}^{\infty} \lambda \, \mbox{ with } \bar{\nu}_{x} = \langle \mathrm{id},\nu_{x} \rangle \mbox{ and }
\bar{\nu}_{x}^{\infty} = \langle \mathrm{id},\nu_{x}^{\infty}\rangle
$$
is the \textbf{barycentre} of $\nu$. If $\lambda (\partial \Omega )=0$, then $\gamma = \bar{\nu}$ and the convergence takes 
place in the stronger sense that $\gamma_{j} \tows \gamma$ narrowly on $\Omega$ (meaning 
$\langle \Phi , \gamma_{j} \rangle \to \langle \Phi , \gamma \rangle$ for all bounded continuous maps $\Phi \colon \Omega \to  \Hi$).

Given a mapping of bounded variation, $u \in \BV(\Omega ,\RN )$, its distributional derivative $Du$ is an $\R^{N \times n}$--valued
Radon measure concentrated on $\Omega$ and hence we can associate
an elementary Young measure $\epsilon_{Du} \in \Yb(\Omega ,\R^{N \times n})$. Writing $Du = \nabla u \Lcal^{n}+D^{s}u$ for the Lebesgue decomposition 
with respect to $\Lcal^n$ we have
$$
  (\epsilon_{Du})_x := \delta_{\nabla u(x)},  \qquad \lambda_{\epsilon_{Du}} := \abs{D^s u},
    \qquad (\epsilon_{Du})_x^\infty := \delta_{p(x)},
$$
where $p := D^s u/\abs{D^s u}$ is short-hand for the Radon-Nikod\'{y}m derivative.

Any Young measure $\nu \in \Yb (\Omega , \R^{N \times n})$ that can be generated by a sequence $(Du_j )$, where
$(u_j)$ is a bounded sequence in $\BV(\Omega ,\RN )$ that $\LL^1$--converges to a map $u \in \BV(\Omega ,\RN )$ is called 
a \textbf{BV gradient Young measure} with underlying deformation $u$. Clearly the underlying deformation $u$ is
locally unique up to an additive constant vector.
If we assume that $\lambda (\partial \Omega) = 0$, 
then $Du$ is the barycentre of $\nu$ and identifying terms according to the Lebesgue decomposition 
$\lambda = a\Lcal^n +b\abs{D^{s}u}+\lambda^{\ast}$ where $\lambda^{\ast}$ is a measure which is singular with respect 
to $\Lcal^n + \abs{D^{s}u}$ we find
\begin{eqnarray}
\nabla u(x) &=& \bar{\nu}_{x} + \bar{\nu}^{\infty}_{x}\frac{\di \lambda}{\di \Lcal^n}(x) \quad \Lcal^{n} \mbox{ a.e. } x \in \Omega \label{ac}\\
\frac{D^{s}u}{\abs{D^{s}u}}(x) &=& b(x)\bar{\nu}^{\infty}_{x} \quad \abs{D^{s}u} \mbox{ a.e. } x \in \Omega \label{singd}\\
0 &=& \bar{\nu}^{\infty}_{x} \quad \lambda^{\ast} \mbox{ a.e. } x \in \Omega . \label{sing} 
\end{eqnarray}
These conditions are clearly necessary for $\nu$ to be a $\BV$ gradient Young measure with barycentre $Du$.
Other necessary conditions follow, loosely speaking, by expressing a relaxation result for signed integrands of linear growth
obtained in \cite{KR10a}, in terms of Young measures. It turns out that these conditions are sufficient for $\nu$
to be a $\BV$ gradient Young measure too. The proof of this is based on a Hahn--Banach argument similar to that employed by 
Kinderlehrer and Pedregal in \cite{KP91,KP94}. Hereby a characterization of $\BV$ gradient Young measures is obtained: they are  
the dual objects to quasiconvex functions in the sense that a set of inequalities akin to Jensen's inequality must hold for 
the measures and all quasiconvex functions of at most linear growth. It can be seen as a nontrivial instance within the abstract frame work
provided by Choquet's theory of function spaces and cones (see \cite{LMNS}). In order to state the result more precisely
we denote by $\Qb$ the class of all quasiconvex functions of at most linear growth, so $f \in \Qb$ when $f \colon \R^{N \times n} \to \R$ 
is quasiconvex and there exists a real constant $c=c_{f}$ such that $\abs{f(\xi )} \leq c \bigl( \abs{\xi}+1 \bigr)$ for all $\xi$.

\begin{theorem}[\cite{KR10b}]\label{gym}
Let $\Omega$ be a bounded Lipschitz domain in $\Rn$. Then, a Young measure $\nu \in \Yb(\Omega , \R^{N \times n})$ 
satisfying $\lambda (\partial \Omega) = 0$ is a $\BV$ gradient Young measure if and only if there exists $u \in \BV(\Omega ,\RN )$ 
such that $Du$ is the barycentre for $\nu$, and, writing $\lambda = \tfrac{\di \lambda}{\di \Lcal^n} \Lcal^n + \lambda^{s}$ for the
Lebesgue--Radon--Nikod\'{y}m decomposition, the following conditions hold for all $f \in \Qb$:
\begin{enumerate}
  \item[(i)]  $\displaystyle f \bigl(\bar{\nu}_{x} + \bar{\nu}^{\infty}_{x}\frac{\di \lambda}{\di \Lcal^n}(x)\bigr) \leq \dprb{f, \nu_x}
    + \dprb{f^\infty , \nu_x^\infty} \, \frac{\di \lambda}{\di \Lcal^n}(x)
    \quad\text{ $\Lcal^n$ a.e.\ $x \in \Omega$,}$
  \item[(ii)] $\displaystyle f^\infty ( \bar{\nu}^{\infty}_{x} ) \leq \dprb{ f^\infty, \nu_x^\infty}$ \quad $\lambda^{s}$ a.e. $x \in \Omega$ .
\end{enumerate}
\end{theorem}
Here the symbol $f^{\infty}$ refers to the (upper) recession function of $f$ as defined at (\ref{urecess}).
The condition $(ii)$ concerning the points $x$ seen by the singular part of the measure $\lambda$ is simply Jensen's inequality 
for the probability measure $\nu_{x}^{\infty}$ and the $1$--homogeneous quasiconvex function $f^{\infty}$. Condition $(i)$ concerning points
$x$ seen by the absolutely continuous part of the measure $\lambda$ is reminiscent of both Jensen's inequality and (\ref{keybound}).
It expresses how oscillation and concentration must be coupled if created by a sequence of gradients. The question of whether this 
characterization can be simplified, and indeed whether condition $(ii)$ is necessary at all, was in fact the initial motivation for the 
work undertaken in the present paper. Inspection of the above characterization yields in particular that a nonnegative finite Radon
measure $\lambda$ on $\cl{\Omega}$ with $\lambda (\partial \Omega )=0$ is the concentration measure for a $\BV$ gradient Young measure
with underlying deformation $u \in \BV (\Omega , \RN )$ if and only if $\lambda^{s} \geq \abs{D^{s}u}$ as measures on $\Omega$. In fact,
it is not hard to see this directly (even without the assumption that $\lambda (\partial \Omega )=0$), and the reason for mentioning
it at this stage is that we would like to emphasize that the concentration measure is more or less arbitrary.

In order to state the the main result of this section we introduce the following notation for special test integrands:
$$
\SQb = \SQb (\R^{N \times n}) = \left\{ f \in \Qb : \, 
\begin{array}{l}
f \mbox{ is Lipschitz, } f(\xi ) \geq \abs{\xi} \mbox{ for all } \xi ,\mbox{ and for some}\\
r_{f} >0 \mbox{ we have } f( \xi )=f^{\infty}(\xi ) \mbox{ when } \abs{\xi} \geq r_{f}  
\end{array}
\right\} .
$$
In terms of these the result is

\begin{theorem}\label{cgym}
Let $\Omega$ be a bounded Lipschitz domain in $\Rn$. Then a Young measure $\nu \in \Yb(\Omega;\R^{N \times n})$ satisfying
$\lambda (\partial \Omega) = 0$ is a $\BV$ gradient Young measure, if and only if $\nu$ has barycentre $\bar{\nu}=Du$ for 
some $u \in \BV (\Omega , \RN )$ and for all $f \in \SQb$ we have
\begin{equation}\label{reduced}
f\bigl( \bar{\nu}_{x} + \bar{\nu}^{\infty}_{x}\frac{\di \lambda}{\di \Lcal^n}(x) \bigr) \leq \dprb{f, \nu_x} +
\dprb{f^{\infty},\nu_{x}^{\infty}}\frac{\di \lambda}{\di \Lcal^n}(x)
\end{equation}
for $\Lcal^n$ almost all $x$.
\end{theorem}
In view of the examples of nonconvex, but quasiconvex positively $1$--homogeneous integrands given in \cite{Mu92} it seems 
that $\SQb$ is close to the minimal class of test integrands we could possibly expect to use in (\ref{reduced}). For instance we note
that it seems to be impossible to reduce the class of test integrands to the smaller one used in the characterization of 
ordinary $\WW^{1,1}$ gradient Young measures ({\em i.e.}, those $\nu$ with $\lambda = 0$) given in \cite{Kr99}.

\begin{proof}
Only the sufficiency part requires a proof, so assume 
$\nu = \bigl( ( \nu_{x})_{x \in \Omega}, \lambda , (\nu_{x}^{\infty})_{x \in \overline{\Omega}} \bigr)$ is a Young measure
such that $\lambda (\partial \Omega) = 0$, $\bar{\nu} = Du$ for some $u \in \BV (\Omega ,\RN )$ and that (\ref{reduced}) 
holds. We will show that conditions $(i)$ and $(ii)$ of Theorem \ref{gym} are satisfied, and start by Lebesgue decomposing the
measure $\lambda$ as $\lambda = a\Lcal^n +b\abs{D^{s}u}+\lambda^{\ast}$, where $\lambda^{\ast}$ is a measure which is 
singular with respect to $\Lcal^n + \abs{D^{s}u}$. Consider first condition $(ii)$, note
$\lambda^{s}=b\abs{D^{s}u}+\lambda^{\ast}$ and fix $f \in \Qb$.
Then the recession function $f^{\infty}$ is a positively $1$--homogeneous and quasiconvex function. Since quasiconvexity
implies rank-one convexity for real--valued functions we can use Corollary \ref{homorc1} whereby we infer that $f^{\infty}$ is convex at
all points of the rank one cone. Therefore Jensen's inequality holds for $f^\infty$ and {\em any} probability measure
with a centre of mass on the rank one cone. According to (\ref{sing}) the probability measure
$\nu_{x}^{\infty}$ has centre of mass at $0$ for $\lambda^{\ast}$ almost all $x \in \Omega$, so $(ii)$ holds $\lambda^{\ast}$
almost everywhere. For the remaining points $x \in \Omega$ seen by $\lambda^s$ we appeal to Alberti's rank-one theorem \cite{Al93}. 
Accordingly the matrix $D^{s}u/\abs{D^{s}u}$ has rank one for $\abs{D^s u}$ almost all $x \in \Omega$, and so by (\ref{singd}) the
probability measure $\nu_{x}^{\infty}$ has centre of mass on the rank one cone for $b\abs{D^{s}u}$ almost all $x \in \Omega$.
Therefore $(ii)$ holds for $b\abs{D^{s}u}$ almost all $x \in \Omega$, and so we have shown that it holds $\lambda^s$ almost 
everywhere in $\Omega$. Next we turn to $(i)$, and start by fixing $f \in \Qb$ with the additional property that $f=f^{\infty}$ 
outside a large ball in matrix space. Because $f^\infty$ in particular must be rank--one convex and positively $1$--homogeneous
we deduce from Corollary \ref{homorc1} that $f^{\infty} \geq \ell$ for some linear function $\ell$ on $\R^{N \times n}$.
Consequently, $f \geq a$ for an affine function $a$ on $\R^{N \times n}$, and defining $g=\abs{\cdot} +(f-a)/\eps$ for $\eps >0$
we record that $g \in \SQb$ so that $(i)$ holds for $g$, and thus also for $\eps g$. By approximation
we deduce that $(i)$ also holds for $f$. The final step is facilitated by the following approximation result.

\begin{lemma}\label{approxi}
Let $f \in \Qb$ and $\delta > 0$. Then $g(\xi ) = g_{\delta}(\xi ):= \max \bigl\{ f(\xi ) , f^{\infty}(\xi ) +\delta \abs{\xi}-\frac{1}{\delta} \bigr\}$
is quasiconvex and for some $s=s(\delta) >0$ we have $g(\xi ) = f^{\infty}(\xi )+\delta \abs{\xi}-\tfrac{1}{\delta}$ for all $\abs{\xi} \geq s$. 
Furthermore, $(g_{\delta})_{\delta \in (0,1)}$ is equi--Lipschitz, and
$$
g_{\delta}(\xi ) \to f(\xi ) \quad \mbox{ and } \quad g_{\delta}^{\infty}(\xi ) \to f^{\infty}(\xi ) 
$$
pointwise in $\xi$ as $\delta \searrow 0$.
\end{lemma}

\begin{proof}[Proof of Lemma \ref{approxi}]
It is clear that $g_{\delta}(\xi ) \to f(\xi )$ as $\delta \searrow 0$ pointwise in $\xi$, that $g_{\delta}$ are quasiconvex 
and, by Lemma \ref{lipschitz}, that $(g_{\delta})$ is equi--Lipschitz.

It remains to find the number $s=s(\delta )$ with the stated property. The rest then follows easily.
Our definition of recession function at (\ref{urecess}) applied to $(f-f^{\infty}+\tfrac{1}{\delta})^{\infty}=0$ yields for given 
$\xi \in \partial \B^{N \times n}$ and $\delta >0$ an $s=s(\xi ,\delta ) >0$ such that
\begin{equation}\label{asympto}
f(t\xi^{\prime}) < f^{\infty}(t\xi^{\prime})+\delta \abs{t\xi^{\prime}}-\frac{1}{\delta}
\end{equation}
for $t \geq s$ and $\xi^{\prime} \in \partial \B^{N \times n}$ with $\abs{\xi-\xi^{\prime}}<1/s$. By compactness of $\partial \B^{N \times n}$
we therefore find an $s = s(\delta ) >0$ such that (\ref{asympto}) holds for all $t\geq s$ and $\xi^{\prime} \in \partial \B^{N \times n}$.
Stated differently we have shown that $f(\xi ) < f^{\infty}(\xi ) + \delta\abs{\xi}-\tfrac{1}{\delta}$ for all $\xi$ with $\abs{\xi} \geq s$. 
But then $g (\xi )=f^{\infty}(\xi ) +\delta \abs{\xi}-\tfrac{1}{\delta}$ for $\abs{\xi} \geq s$, and in particular 
$g^{\infty}(\xi )=f^{\infty}(\xi ) +\delta \abs{\xi}$.
\end{proof}

Fix $f \in \Qb$. Then by the foregoing lemma and the previous step, $(i)$ holds for the integrands $g_{\delta}+\tfrac{1}{\delta}$. But then it also holds for 
each $g_{\delta}$ and so by approximation for $f$.
\end{proof}

\noindent
Department of Mathematics, University of Leipzig
\medskip

\noindent
{\em E-mail address}: Bernd.Kirchheim@math.uni-leipzig.de
\medskip

\noindent
Mathematical Institute, University of Oxford, Andrew Wiles Building, Radcliffe Observatory Quarter,
Woodstock Road, Oxford OX2 6GG, UK
\medskip

\noindent
{\em E-mail address}: kristens@maths.ox.ac.uk

\end{document}